\documentclass[11pt]{article}

\usepackage{amsmath,amssymb,amsthm,amsfonts}
\usepackage{mathtools}
\usepackage{mathrsfs}
\usepackage{bm}
\usepackage{graphicx}
\usepackage{float}
\usepackage{multirow}
\usepackage{booktabs}
\usepackage{longtable}
\usepackage{arydshln}
\usepackage{enumerate}
\usepackage{cite}
\usepackage{hyperref}
\usepackage[a4paper,margin=2.5cm]{geometry}

\hypersetup{
	colorlinks=true,
	linkcolor=blue,
	citecolor=blue,
	urlcolor=blue
}

\allowdisplaybreaks[2]

\numberwithin{equation}{section}

\newtheorem{theorem}{Theorem}[section]
\newtheorem{lemma}[theorem]{Lemma}
\newtheorem{proposition}[theorem]{Proposition}
\newtheorem{corollary}[theorem]{Corollary}
\newtheorem{definition}[theorem]{Definition}

\newtheorem{remark}[theorem]{Remark}

\newcommand{\HH}{\mathbb{H}}
\newcommand{\RR}{\mathbb{R}}

\newcommand{\SR}{\mathcal{SR}}

\title{\Large\bf  Toeplitz Operators on Quaternionic Fock Spaces
	\footnotetext{
		\endgraf Y. Lu was supported by the National Natural Science Foundation of China
		(Grant No. 12031002). C. Zu was supported by the National Natural Science
		Foundation of China (Grant No. 12401151), and the Postdoctoral Researcher
		Foundation of China (Grant No. GZB20240100).
}}

\author{
	Zhaopeng Lin,
	Yufeng Lu,
	Chao Zu\thanks{Corresponding author}
}

\date{}

\begin{document}

  \maketitle
  
  \vspace{-0.8cm}
  
  \begin{center}
  	\begin{minipage}{14cm}\small    {\noindent{\bf Abstract} \quad 
We characterize boundedness and compactness of Toeplitz operators on
quaternionic Fock spaces with  positive measure symbols and  slice-function symbols in \(\mathrm{BMO}^1\). For positive measure symbols,  we derive criteria using normalized reproducing kernels and symmetric box averages, while for slice \(\mathrm{BMO}^1\) symbols, the characterizations rely on the Berezin transform. We further introduce a global quaternionic Fock space \(F_\alpha^p\) to define Toeplitz operators with real-valued measure symbols; this space is built by integrating slice regular functions over all complex slices of \(\mathbb{H}\) and is norm-equivalent to the standard slice-based quaternionic Fock space. In the Hilbert space case \(p=2\), a slice-independent orthogonal projection exists, which allows us to define Toeplitz operators with real-valued measure symbols and slice-function symbols in a unified way.
  			\endgraf
  			\endgraf
  			}
  	\end{minipage}
  \end{center}
   
\section{Introduction}

Quaternionic analysis has developed rapidly in recent decades, giving rise to a function theory that parallels many aspects of one-variable complex analysis. Following Fueter's introduction of the Cauchy--Fueter operator and regular quaternionic functions \cite{fueter1935die}, the modern theory of slice regular 
functions was initiated by Cullen \cite{MR173012}, and further developed by 
Gentili and Struppa \cite{MR2227751,MR2353257}. This theory retains several fundamental tools from complex analysis, including power series expansions and Cauchy integral formulas \cite{MR2752913,MR3013643}. These developments have led to quaternionic analogs of several classical function spaces, such as Hardy spaces \cite{MR3801294,MR3839852,MR3358083}, Bergman spaces \cite{MR34064751,MR3059569,MR3553407,MR3605231}, Fock spaces \cite{MR3587897,MR4011262,MR4205480}, and Bloch spaces \cite{MR3311947,MR4279369}.

Despite these developments, the operator theory on quaternionic Fock spaces remains comparatively underdeveloped. Existing work has
mainly focused on composition operators
\cite{MR4384577,MR4162404,MR4519267,MR4864886,MR4732451,MR4248852}, while the Toeplitz operator
theory on quaternionic Fock spaces has not been systematically developed. 
The aim of this paper is to establish criteria for the boundedness and compactness of Toeplitz operators on quaternionic Fock spaces.

We work with quaternionic Fock spaces of the second kind. In the Hilbert space
case \(p=2\), these spaces have an explicit reproducing kernel and the monomials \(\{ z^n \}\)
form an orthogonal basis. Quaternionic Fock spaces of the first kind were also introduced in
\cite{MR3587897}. However, in the Hilbert space case, the monomials
\(\{z^n\}\) do not form an orthogonal basis, so the corresponding reproducing
kernel does not admit the simple explicit form. For this reason, this 
paper focuses on quaternionic Fock spaces of the second kind.

Let \(\mathbb H\) denote the skew field of quaternions and let
\[
\mathbb S:=\{I\in\mathbb H:I^2=-1\}
\]
be the sphere of imaginary units. For each \(I\in\mathbb S\), the corresponding complex slice is
\[
\mathbb C_I:=\{x+yI:x,y\in\mathbb R\}.
\]
Quaternionic Fock spaces (of the second kind) are defined through fixed-slice
Gaussian norms. More precisely, let \(\alpha>0\), \(0<p<\infty\), and fix
\(I\in\mathbb S\). For a slice regular function \(f\), set
\begin{equation}\label{eq:norm-slice}
	\|f\|_{p,\alpha,I}^p
	= \frac{p \alpha }{2\pi}
	\int_{\mathbb{C}_I}
	\bigl|f(z)\mathrm{e}^{-\alpha|z|^2/2}\bigr|^p
	\,\mathrm{d}m_{2,I}(z),
\end{equation}
where \(m_{2,I}\) denotes the Lebesgue measure on \(\mathbb C_I\). The
corresponding slice quaternionic Fock space is
\[
\mathfrak F_\alpha^p
=
\bigl\{f\in \SR(\mathbb H):\|f\|_{p,\alpha,I}<\infty\bigr\}.
\]
Although this definition is formulated on a fixed slice \(\mathbb C_I\), the
resulting space is independent of the choice of \(I\in\mathbb S\), up to
equivalence of norms.

To define Toeplitz operators with measure symbols on \(\mathbb H\), and to obtain 
a slice-independent projection, we introduce a global Gaussian 
\(L^p\)-framework. Every non-real quaternion has a unique representation
 \(q=x+yI\), with \(I\in\mathbb S\) and \(y>0\). Hence
 \[
 \mathbb H\setminus\mathbb R
 \cong
 \mathbb S\times\mathbb C^+,
 \qquad
 x+yI\longmapsto (I,x+\mathrm i y),
 \]
 where \(\mathbb C^+:=\{x+\mathrm i y:x,y\in\mathbb R,\ y>0\}\) and 
\(\mathrm i\) denotes the usual complex imaginary unit. Through this identification we define a
 measure \(\mathrm dV\) on \(\mathbb H\setminus\mathbb R\), and we extend it to \(\mathbb H\) by declaring \(\mathbb R\) to have \(\mathrm dV\)-measure zero. This gives the global
 slice \(L^p\)-space \(L_{s,\alpha}^p\), consisting of measurable slice functions
 \(f:\mathbb H\to\mathbb H\) such that
 \[
 \|f\|_{p,\alpha}^p
 =
 \frac{p\alpha}{\pi}
 \int_{\mathbb H}
 \bigl|f(q)\mathrm e^{-\alpha |q|^2/2}\bigr|^p
 \,\mathrm dV(q)
 <\infty.
 \]
 The global quaternionic Fock space  \(F_\alpha^p\) consists of all slice regular functions belonging to \(L_{s,\alpha}^p\).
 
We show that, for every \(I\in\mathbb S\),
 \(F_\alpha^p\) coincides with $\mathfrak F_\alpha^p$, with equivalent norms. Thus the
 global norm does not change the class of slice regular functions; rather,
 it provides a slice-independent framework suitable for operator theory. 
 When \(p=2\) (the Hilbert space case), the global space \(F_\alpha^2\) becomes a
 right quaternionic reproducing kernel Hilbert space. 
 Moreover, it admits a slice-independent orthogonal projection $
 P_\alpha:L_{s,\alpha}^2\to F_\alpha^2$.

  To prove the boundedness and compactness results for measure symbols, we first
develop a Fock--Carleson measure theory in the quaternionic setting. The relevant testing sets are the symmetric boxes $S(z,r)$, obtained by taking the axially symmetric completion of Euclidean disks in complex slices.
 These sets are the Fock-space analogue of the
symmetric boxes used in the Carleson measure theory of quaternionic Hardy and
Bergman spaces \cite{MR3664521}. We prove that quaternionic Fock--Carleson
measures and their vanishing versions can be characterized via normalized
kernel tests and the boundedness/vanishing behavior of \(\mu(S(z,r))\),  respectively.

We also employ a slice-wise Berezin transform for slice-function symbols. Given a slice function
\(f\) satisfying the  integrability condition denoted by
\((I_1)\), we define its Berezin transform slice-wise by
\[
(B_\alpha f)(z)
:= \frac{\alpha}{\pi}\int_{\mathbb C_I}\mathrm e^{-\alpha|w-z|^2}f(w)\,\mathrm d m_{2,I}(w),
\qquad z\in\mathbb C_I,\ I\in\mathbb S.
\]
We prove that this definition  yields a
well-defined slice function on \(\mathbb H\).

For slice-function symbols satisfying the condition
\((I_1)\), we define Toeplitz operators initially on the dense subspace $
\mathcal D
:=
\operatorname{span}_{\mathbb H}
\{K_\alpha(\cdot,a):a\in\mathbb H\}$ 
by
\[
T_f g=P_\alpha(f\star g),
\qquad g\in\mathcal D ,
\] where  $\star$ denotes the slice product, which preserves the slice structure.  

We also introduce Toeplitz operators with real-valued measure symbols. If a
real-valued Borel measure \(\mu\) satisfies the integrability condition
\begin{equation}\label{eq:intro-mu2}
\int_{\mathbb H}
|K_\alpha(z,w)|^2\,\mathrm e^{-\alpha|w|^2}\,\mathrm d|\mu|(w)<\infty,
\qquad z\in\mathbb H ,
\end{equation}
then the corresponding Toeplitz operator is initially defined on the dense
subspace \(\mathcal D\) by
\[
T_\mu g(z)=
\int_{\mathbb H}
K_\alpha(z,w)\,g(w)\,\mathrm e^{-\alpha|w|^2}\,\mathrm d\mu(w),
\qquad g\in\mathcal D .
\]
Moreover, if $
\mathrm d\mu=(2\alpha/\pi)f\,\mathrm dV$ 
with \(f\) a real-valued slice function, then the Toeplitz operator defined
from the measure symbol \(\mu\) agrees with the Toeplitz operator defined from
the function symbol \(f\).

We now state the main results obtained in this paper.  The
first two results concern positive measure symbols.  For a positive Borel
measure \(\mu\) on \(\mathbb H\), set
\[
\widehat\mu_r(z):=\frac{\mu(S(z,r))}{\pi r^2},
\qquad z\in\mathbb H,\ r>0.
\]

\begin{theorem}\label{thm:intro-mu-bounded}
Suppose that \(\mu\) is a positive Borel measure on \(\mathbb H\) satisfying 
\eqref{eq:intro-mu2}, and that $
\mu^1:=\mu|_{\mathbb H\setminus\mathbb R}$ 
is a Radon measure on \(\mathbb H\setminus\mathbb R\). Then the following are
equivalent:
\begin{enumerate}[(a)]
\item \(T_\mu\) is bounded on \(F_\alpha^2\);
\item the function
\[
z\longmapsto
\int_{\mathbb H}
\left|k_z(w)\mathrm e^{-\frac{\alpha}{2}|w|^2}\right|^2
\,\mathrm d\mu(w)
\]
is bounded on \(\mathbb H\);
\item \(\widehat\mu_r\) is bounded on \(\mathbb H\) for some, equivalently 
for every, \(r>0\).
\end{enumerate}
\end{theorem}

\begin{theorem}\label{thm:intro-mu-compact}
Suppose that \(\mu\) is a positive Borel measure on \(\mathbb H\)  satisfying 
\eqref{eq:intro-mu2},  
and that \(\mu^1\) is a Radon measure on \(\mathbb H\setminus\mathbb R\). Then
the following are equivalent:
\begin{enumerate}[(a)]
\item \(T_\mu\) is compact on \(F_\alpha^2\);
\item for every \(I\in\mathbb S\),
\[
\int_{\mathbb H}
\left|k_z(w)\mathrm e^{-\frac{\alpha}{2}|w|^2}\right|^2
\,\mathrm d\mu(w)\to0
\qquad \text{as } |z|\to\infty,\ z\in\mathbb C_I;
\]
\item for every \(I\in\mathbb S\), \(\widehat\mu_r(z)\to0\) as
\(|z|\to\infty\), \(z\in\mathbb C_I\), for some, equivalently 
for every,
\(r>0\).
\end{enumerate}
\end{theorem}
The next two results concern slice-function symbols in \(\mathrm{BMO}^1\).
Here \(\mathrm{BMO}^1\) denotes the class of functions on $\mathbb H$ obtained by requiring uniformly
bounded mean oscillation on the slices \(\mathbb C_I\); the precise definition is
given in Section~\ref{sec:berezin-bmo}. For a slice function \(f\) locally integrable on every slice, we define its slice-wise disk average by
\[
\widehat f_r(z)
:=
\frac{1}{\pi r^2}
\int_{B_I(z,r)} f_I(w)\,\mathrm d m_{2,I}(w),
\qquad z\in\mathbb C_I,\ I\in\mathbb S.
\]
We prove that this definition  gives a
well-defined slice function on \(\mathbb H\).  We write \(C_0(\mathbb H)\) for the space of slice functions
\(f:\mathbb H\to\mathbb H\) such that each restriction \(f_I\) is continuous on
\(\mathbb C_I\) and vanishes at infinity on \(\mathbb C_I\), for every
\(I\in\mathbb S\).
\begin{theorem}\label{thm:intro-bmo-bounded}
Let \(f\in \mathrm{BMO}^1\) be a slice function. Then the following are
equivalent:
\begin{enumerate}[(a)]
\item \(T_f\) is bounded on \(F_\alpha^2\);
\item \(B_\alpha f\in L_s^\infty(\mathbb H,\mathrm dV)\);
\item \(B_\beta f\in L_s^\infty(\mathbb H,\mathrm dV)\) for some, equivalently 
for every, \(\beta>0\);
\item \(\widehat f_r\in L_s^\infty(\mathbb H,\mathrm dV)\) for some, equivalently 
for every, \(r>0\).
\end{enumerate}
\end{theorem}

\begin{theorem}\label{thm:intro-bmo-compact}
Let \(f\in \mathrm{BMO}^1\) be a slice function, then \(T_f\) is compact on
\(F_\alpha^2\) if and only if $
B_\alpha f \in C_0(\mathbb H)$.
\end{theorem} 

Two difficulties are specific to the quaternionic setting. The first is
structural. There is no canonical Toeplitz projection on \(\mathbb H\) obtained
directly from a fixed complex slice. The global Fock-space framework and the
projection \(P_\alpha\) resolve this issue and make the function-symbol and
measure-symbol definitions compatible. The second difficulty is algebraic. The noncommutativity of slice product makes the resulting Toeplitz theory substantially more intricate than in the complex setting.  In the complex Fock-space setting, the
pointwise product makes the Toeplitz correspondence linear in the symbol. In the
quaternionic setting, however, Toeplitz operators are defined by means of the
slice product.  Consequently, scalar multiplication of the symbol does not behave
as in the complex case. Indeed, for \(a\in\mathbb H\),
right multiplication of the symbol gives
\[
T_{f\star a}=T_f\circ L_a,\qquad L_ag:=a\star g,
\]
rather than ordinary scalar
multiplication. One cannot expect a direct analogue of the classical adjoint relation $T_f^*=T_{\overline f}$ in general.
 
The paper is organized as follows.
Section~\ref{sec:preliminaries} recalls slice functions, slice products, and
quaternionic Fock spaces, and introduces the global Fock framework. Section~\ref{s:4} establishes the quaternionic
Fock--Carleson embedding theorem and the corresponding vanishing theorem. Section~\ref{sec:berezin-bmo}
studies the Berezin transform of slice functions and introduces the slice
\(\mathrm{BMO}^p\) classes. Section~\ref{sec:toeplitz} proves the boundedness and compactness criteria for
positive measure symbols, nonnegative real-valued slice-function symbols, and
slice \(\mathrm{BMO}^1\) symbols.
\medskip 

Throughout the paper, the notation $a\lesssim b$ means that there exists a constant $C>0$ such that $a\le Cb$, while $a\approx b$ means that both $a\lesssim b$ and $b\lesssim a$ hold. Unless explicitly stated otherwise, $C$ denotes a positive constant that may vary from occurrence to occurrence.
\section{Preliminaries}\label{sec:preliminaries}
\subsection{Slice functions}
Recall that the quaternionic field is denoted by
\[
\mathbb{H} = \{x_0 + x_1 i + x_2 j + x_3 k : x_l \in \mathbb{R}, ~ 0 \leq l \leq 3\},
\]
which is a four-dimensional noncommutative algebra generated by the imaginary units 
$i, j, k$, subject to the relations
\[
i^2 = j^2 = k^2 = -1, \quad ij = -ji = k, \quad jk = -kj = i, \quad ki = -ik = j.
\]

For a quaternion $q \in \mathbb{H}$, its Euclidean norm is defined by
\[
|q| = \sqrt{q \,\overline{q}} = \sqrt{\overline{q} \, q} = \big(\sum_{l=0}^3 x_l^2\big)^{1/2},
\]
where
\[
\overline{q} = x_0 - (x_1 i + x_2 j + x_3 k)
\]
denotes the quaternionic conjugate of $q$.

For any $q\in\mathbb{H}$, write
\[
q=\operatorname{Re}(q)+\operatorname{Im}(q).
\]
Set
\[
x:=\operatorname{Re}(q)\in\mathbb{R},
\qquad
y:=|\operatorname{Im}(q)|\ge 0.
\]
If $y\neq 0$, define
\[
I_q:=\frac{\operatorname{Im}(q)}{|\operatorname{Im}(q)|}\in\mathbb{S},
\]
while if $y=0$ (i.e.\ $q\in\mathbb{R}$), let $I_q$ be an arbitrary element of $\mathbb{S}$. Then
\[
q=x+yI_q.
\]

\begin{definition} 
	\label{def:slice-regular}
	A function \(f:\mathbb H\to\mathbb H\) is called a (left) slice function if there exist
	functions \(f_0,f_1:\mathbb R^2\to\mathbb H\) such that
	\[
	f(x+yI)=f_0(x,y)+I f_1(x,y),
	\qquad x,y\in\mathbb R,\ I\in\mathbb S,
	\]
	with the parity conditions
	\[
	f_0(x,-y)=f_0(x,y),
	\qquad
	f_1(x,-y)=-f_1(x,y).
	\]
  
	The slice function \(f\) is called slice intrinsic if \(f_0\) and \(f_1\) are real-valued.
	Equivalently,
	\[
	f(\mathbb C_I)\subseteq \mathbb C_I
	\qquad \text{for every } I\in\mathbb S .
	\] 
 If, in addition,  \(f_0\) and \(f_1\) are continuously differentiable and satisfy
	the Cauchy--Riemann system
	\[
	\frac{\partial f_0}{\partial x}
	-
	\frac{\partial f_1}{\partial y}=0,
	\qquad
	\frac{\partial f_0}{\partial y}
	+
	\frac{\partial f_1}{\partial x}=0,
	\]
	then \(f\) is called left slice regular. The class of left slice regular functions
	on \(\mathbb H\) is denoted by \(\SR(\mathbb H)\).
\end{definition}

\begin{lemma}\label{eq:11}
 \cite[Lemma 1.3]{MR3013643} (Splitting Lemma). If $f$ is slice  regular on $\mathbb{H}$, then for any
$I,J\in  \mathbb{S}$ with $I\perp J$, there exist two holomorphic functions $F,G:\mathbb{C}_I \to\mathbb{C}_I$
such that 
\begin{equation}\label{eq:split}
    f_I(z)=F(z)+G(z)J \ \text{for all } z\in \mathbb C_I.
\end{equation} 
\end{lemma}

\begin{remark}\label{rem:splitting_vs_slice}
Since $
\mathbb H=\mathbb C_I\oplus \mathbb C_IJ$, 
every \(\mathbb H\)-valued function \(h\) on \(\mathbb C_I\) admits a unique
decomposition
\[
h(z)=F(z)+G(z)J,\qquad z\in\mathbb C_I,
\]
where \(F,G:\mathbb C_I\to\mathbb C_I\). The Splitting Lemma asserts more:
if \(h=f_I\) is the restriction of a slice regular function \(f\), then the
components \(F\) and \(G\) are holomorphic on \(\mathbb C_I\).
\end{remark}

\begin{theorem} \label{eq:12}
  \cite[Corollary 1.17]{MR3013643} (Representation Formula) Let $f$ be a slice regular function defined on $\mathbb{H}$. Fix $J\in\mathbb S$. Then for every $x,y\in\mathbb R$ and every $I\in\mathbb S$,  
\begin{equation}\label{rep}
\begin{split}
f(x + yI) &= \frac{1}{2}\Big[f(x + yJ) + f(x - yJ)\Big] + I\frac{1}{2}\Big[J[f(x - yJ) - f(x + yJ)]\Big] \\
&= \frac{1}{2}(1 - IJ)f(x + yJ) + \frac{1}{2}(1 + IJ)f(x - yJ).
\end{split}
\end{equation}
Moreover, for every $J\in\mathbb S$ and all $x,y\in\mathbb R$, there exist functions $a_f,b_f$, independent of $J$, such that
\begin{equation}\label{alpha beta}
   a_f(x,y) =\frac{1}{2}\Big[f(x+yJ)+f(x-yJ)\Big]  , \quad b_f(x,y)=\frac{1}{2}\Big[J[f(x-yJ)-f(x+yJ)]\Big].
\end{equation}
 These functions satisfy
\begin{equation}\label{eq:parity-ab}
a_f(x,-y)=a_f(x,y),
\qquad
b_f(x,-y)=-\,b_f(x,y),
\qquad (x,y)\in\mathbb R^2.
\end{equation}
\end{theorem}
\begin{theorem}\label{thm 2.4} \cite[Lemma 1.21]{MR3013643} (Extension Theorem)
Let $J\in\mathbb{S}$ and let $\psi:\mathbb{C}_J\to\mathbb{H}$ satisfy $\overline{\partial}_J \psi=0$ on $\mathbb{C}_J$.
Then the function
\[
\operatorname{ext}(\psi)(x+yI)
=\frac12\bigl(\psi(x+yJ)+\psi(x-yJ)\bigr)
+I\frac12\Bigl(J\bigl(\psi(x-yJ)-\psi(x+yJ)\bigr)\Bigr)
\]
defines a slice regular function on $\mathbb{H}$ whose restriction to $\mathbb{C}_J$ equals $\psi$.
Moreover, this extension is unique among slice regular functions on $\mathbb{H}$ with the same restriction to $\mathbb{C}_J$.
\end{theorem}
In fact, both the Representation Formula and the Extension Theorem remain valid on axially symmetric domains\cite{MR2555912}, and they extend to several broader classes of domains as well\cite{MR4634680}.

 We next recall the slice product for  slice functions; see
 \cite{MR3632548}.
 
 \begin{definition} \label{def:slice-product}
 	Let \(f,g:\mathbb H\to\mathbb H\) be  slice functions. Write
 	\[
 	f(x+yI)=f_0(x,y)+I f_1(x,y),
 	\qquad
 	g(x+yI)=g_0(x,y)+I g_1(x,y),
 	\]
 	where \(f_0,f_1,g_0,g_1:\mathbb R^2\to\mathbb H\) satisfy the parity conditions. 
 	The  slice product of \(f\) and \(g\), denoted by \(f\star g\), is the  slice function defined by
 \begin{equation}\label{eq:star-stem-form} (f\star g)_I(x+yI) = (f_0g_0-f_1g_1)(x,y)+I(f_0g_1+f_1g_0)(x,y). \end{equation}
 Moreover,	if \(f,g\in \SR(\mathbb H)\), then \(f\star g\in \SR(\mathbb H)\).
 \end{definition}
 
We shall repeatedly use the following slice-wise formula for the slice product.

\begin{proposition}\label{prop:slice_star_formula}
	Let \(f\) and \(g\) be slice functions on \(\mathbb H\). Fix
	\(I\in\mathbb S\) and choose \(J\in\mathbb S\) with \(J\perp I\). Write
	\[
	f_I(z)=F_0(z)+F_1(z)J,
	\qquad
	g_I(z)=G_0(z)+G_1(z)J,
	\qquad z\in\mathbb C_I,
	\]
	where \(F_0,F_1,G_0,G_1:\mathbb C_I\to\mathbb C_I\). Then
	\[
	(f\star g)_I(z)
	=
	\bigl[F_0(z)G_0(z)-F_1(z)\overline{G_1(\bar z)}\bigr]
	+
	\bigl[F_0(z)G_1(z)+F_1(z)\overline{G_0(\bar z)}\bigr]J ,
	\qquad z\in\mathbb C_I .
	\]
\end{proposition}

\begin{proof}
This is the standard slice-wise expression of the slice product. It follows by 
writing the stem components in the decomposition 
\(\mathbb H=\mathbb C_I\oplus\mathbb C_IJ\) and using 
\(Ju=\overline u\,J\) for \(u\in\mathbb C_I\); see 
\cite{colombo2016entire}.
\end{proof}

\begin{remark}\label{rem:star_vs_pointwise}
(a) The $\star$-product is associative and distributive but, in general, not commutative.

(b) The $\star$-product  preserves the slice property,  whereas the pointwise product need not preserve the slice property.
  
\end{remark}


The following lemma gives several equivalent characterizations of slice intrinsic functions. In \cite{colombo2016entire}, slice intrinsicity is described in terms of preserving every slice. Here we show that it is enough to verify the property on finitely many slices.   

\begin{lemma}\label{lemma:2.5}
Suppose $f$ is a slice function on $\mathbb{H}$. The following statements are equivalent:
\begin{enumerate}
\item[(a)] $f$ is slice intrinsic.

\item[(b)] There exist $I_1,I_2\in\mathbb{S}$ with $I_1\neq \pm I_2$ such that
\[
f\bigl(\mathbb{C}_{I_\ell}\bigr)\subseteq \mathbb{C}_{I_\ell},
\qquad \ell=1,2.
\]

\item[(c)] There exist three imaginary units $I_1,I_2,I_3\in\mathbb{S}$ which are linearly independent in $\operatorname{Im}\mathbb{H}\cong\mathbb{R}^3$ such that
\[
f(\overline{z})=\overline{f(z)}
\qquad\text{for every } z\in \mathbb{C}_{I_\ell},  \ell =1,2,3.
\]

\item[(d)] There exists $I_0\in \mathbb{S}$ such that
\[
f\bigl(\mathbb{C}_{I_0}\bigr)\subseteq \mathbb{C}_{I_0}
\quad\text{and}\quad
f(\overline{z})=\overline{f(z)}\ \text{for every } z\in \mathbb{C}_{I_0}.
\]
\end{enumerate}
\end{lemma}

\begin{proof}
Since $f$ is a slice function, there exist functions $a,b:\mathbb{R}^2\to\mathbb{H}$ such that
\[
a(x,-y)=a(x,y),\qquad b(x,-y)=-b(x,y),
\]
and for every $I\in\mathbb{S}$,
\[
f(x+yI)=a(x,y)+I\,b(x,y).
\]
For \((x,y)\in\mathbb{R}^2\), we write
\[
a:=a(x,y),\qquad b:=b(x,y).
\]
 
\noindent\textbf{(a)$\Rightarrow$(b), (a)$\Rightarrow$(c), (a)$\Rightarrow$(d).}
If $f$ is slice intrinsic, then $a$ and $ b$ are real-valued. Hence $f(\mathbb{C}_I)\subseteq\mathbb{C}_I$ for every $I$ and
$f(\overline z)=\overline{f(z)}$ for all $z\in\mathbb{H}$. In particular, (b), (c), (d) hold.
 
\noindent\textbf{(b)$\Rightarrow$(a).} 
For $\ell=1,2$ put $z_\ell:=x+yI_\ell\in\mathbb{C}_{I_\ell}$, so $\overline{z_\ell}=x-yI_\ell\in\mathbb{C}_{I_\ell}$.
Then
\[
f(z_\ell)=a+I_\ell b\in\mathbb{C}_{I_\ell},\qquad f(\overline{z_\ell})=a-I_\ell b\in\mathbb{C}_{I_\ell}.
\]
Hence
\[
a=\frac{f(z_\ell)+f(\overline{z_\ell})}{2}\in\mathbb{C}_{I_\ell},
\qquad
I_\ell b=\frac{f(z_\ell)-f(\overline{z_\ell})}{2}\in\mathbb{C}_{I_\ell},
\]
and since $\mathbb{C}_{I_\ell}$ is a (real) subalgebra containing $I_\ell$,
\[
b=-I_\ell\,(I_\ell b)\in\mathbb{C}_{I_\ell}.
\]
Therefore $a,b\in \mathbb{C}_{I_1}\cap\mathbb{C}_{I_2}$.
If $I_1\neq\pm I_2$, then $\mathbb{C}_{I_1}\cap\mathbb{C}_{I_2}=\mathbb{R}$, so $a,b\in\mathbb{R}$.
Since $(x,y)$ was arbitrary, we obtain that $f$ is slice intrinsic.
 
\noindent\textbf{(c)$\Rightarrow$(a).}  
For each $\ell\in\{1,2,3\}$ set $z_\ell:=x+yI_\ell\in\mathbb{C}_{I_\ell}$.
By parity of $a,b$,
\[
f(\overline{z_\ell})=f(x-yI_\ell)=a-I_\ell b.
\]
On the other hand, we have
\[
\overline{f(z_\ell)}=\overline{a+I_\ell b}=\overline{a}+\overline{I_\ell b}.
\]
Since $f(\overline{z_\ell})=\overline{f(z_\ell)}$ we get 
\[
a-\overline{a}=I_\ell b+\overline{I_\ell b}.
\]
Hence
\[
\operatorname{Im}(a)=0
\qquad\text{and}\qquad
\operatorname{Re}(I_\ell b)=0,
\qquad \ell=1,2,3.
\]
Thus $a\in\mathbb{R}$.

Write
\[
b=b_0+\mathbf b,
\qquad b_0\in\mathbb{R},\ \mathbf b\in\operatorname{Im}\mathbb{H}.
\]
Then for each $\ell$,
\[
\operatorname{Re}(I_\ell b)=\operatorname{Re}(I_\ell\mathbf b).
\]
Identifying $\operatorname{Im}\mathbb H$ with $\mathbb R^3$, and using the standard identity
\[
uv=-\langle u,v\rangle+u\times v,\qquad u,v\in \operatorname{Im}\mathbb H,
\]
we obtain
\[
\operatorname{Re}(I_\ell\mathbf b)=-\langle I_\ell,\mathbf b\rangle,
\]
where $\langle\cdot,\cdot\rangle$ denotes the Euclidean inner product on $\operatorname{Im}\mathbb{H}$. Since
\[
\operatorname{Re}(I_\ell b)=0,
\qquad \ell=1,2,3,
\]
and $I_1,I_2,I_3$ are linearly independent, it follows that
 $
\mathbf b=0$. 
Hence $b\in\mathbb{R}$. Since $(x,y)$ was arbitrary, it follows that $f$ is slice intrinsic.
 
\noindent\textbf{(d)$\Rightarrow$(a).}
 Let $z:=x+yI_0\in\mathbb C_{I_0}$. As in the proof of \((b)\Rightarrow(a)\), the condition $
f(\mathbb C_{I_0})\subseteq \mathbb C_{I_0}$ 
implies that
\[
a,b\in \mathbb C_{I_0}.
\]
Write
\[
a=a_0+a_1I_0,\qquad b=b_0+b_1I_0,
\qquad a_0,a_1,b_0,b_1\in\mathbb R.
\]
Since
\[
f(\bar z)=\overline{f(z)},
\]
arguing as in \((c)\Rightarrow(a)\) gives
\[
\operatorname{Im}(a)=0,
\qquad
\operatorname{Re}(I_0b)=0.
\]
Because $a,b\in\mathbb C_{I_0}$, these identities reduce to
\[
a_1=0,\qquad b_1=0.
\]
Hence $a,b\in\mathbb R$. Since $(x,y)\in\mathbb R^2$ was arbitrary, it follows that $f$ is slice intrinsic.

\end{proof}

\begin{remark}\label{rem:one-slice-not-intrinsic}
A slice function preserving a single slice need not be slice intrinsic. Indeed, fix $I\in\mathbb S$ and define
\[
f(z)\equiv I,\qquad z\in\mathbb H.
\]
Then $f$ is a slice function and $f(\mathbb C_I)\subseteq \mathbb C_I$. However, if $J\in\mathbb S$ with $J\neq\pm I$, then $I\notin\mathbb C_J$, so $f(\mathbb C_J)\nsubseteq \mathbb C_J$. Hence $f$ is not slice intrinsic.

    Likewise, a slice function satisfying  $f(\overline z)=\overline{f(z)}$ for two slices need not be intrinsic. For example, define
\[
f(x+yI):=I(yi),
\qquad x,y\in\mathbb R,\ I\in\mathbb S.
\]
Then $f$ is a slice function with stem function
\[
a(x,y)\equiv 0,\qquad b(x,y)=yi.
\]
Moreover, for every $z\in\mathbb C_j\cup\mathbb C_k$ one has $
f(\overline z)=\overline{f(z)}$. 
Indeed, if $z=x+yj\in\mathbb C_j$, then
\[
f(z)=j(yi)=-yk,\qquad f(\overline z)=j((-y)i)=yk=\overline{f(z)}.
\]
Similarly, if $z=x+yk\in\mathbb C_k$, then
\[
f(z)=k(yi)=yj,\qquad f(\overline z)=k((-y)i)=-yj=\overline{f(z)}.
\]
However,
\[
f(\mathbb C_j)\nsubseteq \mathbb C_j,
\]
since $f(x+yj)=-yk\in\mathbb C_k$ for $y\neq 0$. Hence $f$ is not slice intrinsic.
\end{remark}

 \subsection{Quaternionic Fock Spaces} 

We use the following notation throughout the paper. The symbol \(F_\alpha^p\)
denotes the global quaternionic Fock space on \(\mathbb H\). The symbol
\(\mathfrak F_\alpha^p\) denotes the slice-defined quaternionic Fock space; its
definition uses a fixed slice, but the resulting class of slice regular functions
is independent of the chosen slice up to equivalence of norms. For a fixed
\(I\in\mathbb S\), \(\mathcal F_\alpha^p(\mathbb C_I)\) denotes the classical
\(\mathbb C_I\)-valued complex Fock space on the plane \(\mathbb C_I\).

\begin{definition}
   Let $I\in\mathbb S$, $\alpha>0$, and $0<p<\infty$. For a measurable function  $f:\mathbb H\to\mathbb H$, define
\[
\|f\|_{p,\alpha,I}^p
:=
\frac{p \alpha }{2\pi}\int_{\mathbb C_I}\bigl|f(z)\mathrm e^{-\alpha|z|^2/2}\bigr|^p\,\mathrm d m_{2,I}(z),
\]
where $\mathrm d m_{2,I}$ denotes the Lebesgue measure on $\mathbb C_I$.
    
    For $p=\infty$, define
\[
\|f\|_{\infty,\alpha,I}
:=
\operatorname*{ess\,sup}_{z\in\mathbb C_I}\bigl|f(z)\mathrm e^{-\alpha|z|^2/2}\bigr|.
\]

   We then define the slice quaternionic Fock space
\[
\mathfrak F_{\alpha }^p
:=
\bigl\{f\in \SR(\mathbb H):\|f\|_{p,\alpha,I}<\infty\bigr\}.
\]

\end{definition} 
 It is known that this definition is independent of the choice of  \(I\in\mathbb S\), up to
  equivalence of norms; see, for example, \cite[Proposition~3.8]{MR3587897}. 
\begin{remark}\label{rem:slice_Fock_norm_independent}
	Fix \(I\in\mathbb S\). With
	\[
	\mathrm d\lambda_{\alpha,I}(z)
	=
	\frac{\alpha}{\pi}\mathrm e^{-\alpha|z|^{2}}\mathrm d m_{2,I}(z),
	\]
	the space \(\mathfrak F_{\alpha}^{2}\) becomes a right quaternionic Hilbert
	space when equipped with
	\[
	\langle f,g\rangle_{\alpha,I}
	:=
	\int_{\mathbb C_I}\overline{g(z)}\,f(z)\,
	\mathrm d\lambda_{\alpha,I}(z).
	\]
	The inner product is right linear in the first variable and conjugate linear in
	the second.
	
	Moreover, the corresponding \(L^2\)-norm is independent of the choice of the
	slice for every measurable slice function. More precisely, if \(f\) is a
	measurable slice function and \(\|f\|_{2,\alpha,I}<\infty\) for some
	\(I\in\mathbb S\), then
	\[
	\|f\|_{2,\alpha,I}=\|f\|_{2,\alpha,J},
	\qquad I,J\in\mathbb S .
	\]
	Indeed, writing
	\[
	f(x+yJ)=a_f(x,y)+Jb_f(x,y),
	\]
	we have
	\[
	|a_f(x,y)|^2+|b_f(x,y)|^2
	=
	\frac12\Bigl(|f(x+yJ)|^2+|f(x-yJ)|^2\Bigr),
	\]
	and the left-hand side is independent of \(J\). Multiplying by
	\(\mathrm e^{-\alpha(x^2+y^2)}\) and integrating over \(\mathbb R^2\), we obtain
	\[
	\int_{\mathbb R^2}
	\bigl(|a_f(x,y)|^2+|b_f(x,y)|^2\bigr)
	\mathrm e^{-\alpha(x^2+y^2)}\,\mathrm d x\,\mathrm d y
	=
	\int_{\mathbb R^2}
	|f(x+yJ)|^2
	\mathrm e^{-\alpha(x^2+y^2)}\,\mathrm d x\,\mathrm d y,
	\]
	where the equality follows from the change of variables \(y\mapsto -y\).
	Therefore, up to the normalizing constant \(\alpha/\pi\), the right-hand side
	is precisely \(\|f\|_{2,\alpha,J}^2\), and hence it is independent of \(J\).
	In particular, for \(f\in \SR(\mathbb H)\), the Fock norm
	\(\|f\|_{2,\alpha,I}\) is independent of \(I\).
\end{remark}
 
\begin{proposition}\label{prop:3.6}
\cite[Theorem~3.10]{MR3587897}
Let $\alpha>0$ and fix $I\in\mathbb{S}$. Then $\bigl(\mathfrak{F}^2_{\alpha },\langle\cdot,\cdot\rangle_{\alpha,I}\bigr)$
is a quaternionic reproducing kernel Hilbert space. Its reproducing kernel is
\[
K_{\alpha}(z,w)=\sum_{n=0}^{\infty}\frac{\alpha^{n}}{n!}\, z^{n}\,\overline{w}^{\,n}:=  \mathrm{e}^{\alpha z\overline{w}}_{\star},
\qquad z,w\in\mathbb{H}.
\]
\end{proposition} 
There is a canonical bijection between $\mathbb H\setminus\mathbb R$ and $\mathbb S\times\mathbb C^+$, given by
\[
x+yI \longmapsto (I,x+\mathrm i y),
\]
where \(\mathbb C^+=\{x+\mathrm i y:x,y\in\mathbb R,\ y>0\}\) and
\(\mathrm i\) denotes the usual complex imaginary unit.

Via this identification, we endow $\mathbb H\setminus\mathbb R$ with the product measure
\[
\mathrm dV(z):=\mathrm d m_{2,I}(z)\,\mathrm d\sigma(I),\qquad z=x+yI\in\mathbb C_I,\ y>0.
\]
Here $\sigma$ is the normalized surface measure on $\mathbb{S}$. We extend $\mathrm{d}V$ to $\mathbb H$ by declaring the real axis to have $\mathrm{d}V$-measure zero; hence integrals over
$\mathbb H$ are understood as integrals over $\mathbb H\setminus\mathbb R$.
We further define the Gaussian probability measure on $\mathbb{H}$ by
\[
\mathrm{d}\lambda_\alpha(z):=\frac{2\alpha}{\pi}\mathrm{e}^{-\alpha|z|^2}\,\mathrm{d}V(z).
\]
With this choice of constant, $\lambda_\alpha$ is a probability measure on $\mathbb{H}\setminus\mathbb{R}$ and the real axis has $\lambda_\alpha$-measure zero. 

\begin{definition}
Let $\alpha>0$ and $0<p<\infty$. For a measurable function $f:\mathbb H\to\mathbb H$, define
    \begin{equation}
        \|f\|_{p,\alpha}^p = \frac{p \alpha }{\pi} \int_{\mathbb{H}} |f(z)  \mathrm{e}^{-\alpha |z|^2/2} |^p \, \mathrm{d}V(z).
    \end{equation} 
    
    For $p = \infty$, define
    $$ \|f\|_{\infty,\alpha} = \operatorname*{ess\,sup}\left\{|f(z)| \mathrm{e}^{-\alpha |z|^2/2} : z \in \mathbb{H} \right\}. $$
\end{definition} 

\begin{definition}
	Let \(\alpha>0\) and \(0<p\le \infty\). Define \(L^p_{s,\alpha}\) to be the
	space of measurable slice functions \(f:\mathbb H\to\mathbb H\) satisfying $
	\|f\|_{p,\alpha}<\infty$. 
	The global quaternionic Fock space is
	\[
	F_\alpha^p:=L^p_{s,\alpha}\cap \SR(\mathbb H).
	\]
\end{definition}
\begin{remark}
 The space $F_\alpha^2$ is a right quaternionic Hilbert space with inner product $\langle \cdot ~,~\cdot \rangle_{\alpha}$ defined by
  $$ \langle f , g \rangle_{\alpha} = \int_{\mathbb{H}} \overline{g(z)} f(z) \, \mathrm{d}\lambda_{\alpha}(z) $$ for $f, g \in F^2_\alpha$.
 
\end{remark}


\begin{theorem}\label{thm dengju}
Let \(\alpha>0\) and \(0<p<\infty\). Let \(f\) be a measurable slice function
on \(\mathbb H\). Then, for every fixed \(I_0\in\mathbb S\),
 $
\|f\|_{p,\alpha}
\approx
\|f\|_{p,\alpha,I_0}$. 
Consequently,
\[
f\in L^p_{s,\alpha}
\quad\Longleftrightarrow\quad
\|f\|_{p,\alpha,I_0}<\infty
\]
for some, equivalently for every, \(I_0\in\mathbb S\).

In the  case \(p=2\), one has $
\|f\|_{2,\alpha}
=
\|f\|_{2,\alpha,I_0}, I_0\in\mathbb S$. 
Moreover, for \(f,g\in L^2_{s,\alpha}\),
\[
\langle f,g\rangle_{\alpha}
=
\langle f,g\rangle_{\alpha,I_0}.
\]
\end{theorem}

\begin{proof} 
Fixed $I_0\in\mathbb S$, a direct computation based on the definition of $\mathrm dV$ yields
\begin{equation}
\begin{split}
||f||^p_{p,\alpha} &= \frac{p\alpha}{\pi} \int_\mathbb{H} |f(z) \mathrm{e}^{-\frac{\alpha}{2}|z|^2}|^p \, \mathrm{d}V(z) \\
&= \frac{p \alpha }{\pi} \int_\mathbb{S} \int_0^\pi \int_0^\infty |f(r\mathrm{e}^{I\theta}) \mathrm{e}^{-\frac{\alpha}{2} r^2}|^p r \, \mathrm{d}r \, \mathrm{d}\theta \, \mathrm{d} \sigma(I) \\
&= \frac{p \alpha }{2\pi} \int_\mathbb{S} \int_0^{2\pi} \int_0^\infty |f(r\mathrm{e}^{I\theta}) \mathrm{e}^{-\frac{\alpha}{2} r^2}|^p r \, \mathrm{d}r \, \mathrm{d}\theta \, \mathrm{d} \sigma(I) \\
&= \int_\mathbb{S} \|f\|^p_{p,\alpha,I} \, \mathrm{d} \sigma(I) \\
&\approx \|f\|_{p,\alpha,I_0}^p,
\end{split}
\end{equation}
where the last equivalence follows from the equivalence of the slice norms for
different imaginary units; see \cite[Proposition~3.8]{MR3587897}.

In the  case $p=2$,  by  Remark \ref{rem:slice_Fock_norm_independent}, the slice norm $\|\cdot\|_{2,\alpha,I}$ is independent of the choice of $I \in \mathbb S$. Hence, 
 \[
\| f \|^2_{2,\alpha}= \int_{\mathbb{S}} \| f \|^2_{2,\alpha,I} \mathrm{d} \sigma(I) = \| f \|_{2,\alpha,I_0}^2.
  \] 
Applying the quaternionic polarization identity
\cite[p.~023503-11]{MR3450567}, we obtain 
\[
\langle f,g\rangle_{\alpha}=\langle f,g\rangle_{\alpha,I_0}.
\]
Thus the global and fixed-slice inner products agree on \(L^2_{s,\alpha}\).
\end{proof}
By Theorem \ref{thm dengju}, we obtain
\begin{corollary}\label{cor:Fock-global-slice-equivalence}
	Let \(\alpha>0\) and \(0<p<\infty\). Then  the global quaternionic Fock space $F_\alpha^p$ and the slice quaternionic Fock space $\mathfrak F_{\alpha }^p$ coincide as sets and have equivalent norms. Moreover,  the spaces $F_\alpha^2$ and $\mathfrak F_{\alpha }^2$ are isometrically isomorphic.
\end{corollary}
Since $F_\alpha^2$ is a closed subspace of $L_{s,\alpha}^2$, there exists an orthogonal projection
\[
P_\alpha:L_{s,\alpha}^2\to F_\alpha^2,
\]
given by
\[
P_\alpha h(z)
=
\int_{\mathbb H}\mathrm e_*^{\alpha z\bar w}\,h(w)\,
\mathrm d\lambda_\alpha(w),
\qquad h\in L_{s,\alpha}^2 .
\]
For bounded slice function $f$, we define
 \[
 T_f g:=P_\alpha(f\star g),\qquad g\in F_\alpha^2.
 \]
 We call $T_f$ the Toeplitz operator on $F^2_\alpha$ with symbol $f$.

On the other hand, motivated by \cite{MR4162404}, one may also define Toeplitz
operators first on a fixed slice and then extend them to \(\mathbb H\). For a bounded slice function \(f\), define on the slice
\(\mathbb C_I\)
\[
\bigl(T_f^{(I)}g\bigr)_I(z)
:=\langle (f\star g)_I , K_{\alpha}(\cdot,z) \rangle_{\alpha,I}=
\int_{\mathbb C_I} e_*^{\alpha z\bar w}\,\bigl(f \star g  \bigr)_I(w)\,\mathrm d\lambda_{\alpha,I}(w) ,
 \qquad z\in\mathbb C_I,~~g\in F_\alpha^2,
\]  
and then extend to $\mathbb H$ by the (unique) slice regular extension
\[
\bigl(T^{(I)}_f g\bigr)(q)
:=\operatorname{ext}\!\left(\,T^{(I)}_f g\big|_{\mathbb C_I}\right)(q),\qquad q\in\mathbb H .
\] 

We claim that the global and fixed-slice constructions agree. Indeed, by Theorem~\ref{thm dengju}, for every $g\in F_\alpha^2$ and    $z\in \mathbb{C}_I $ one has
\begin{equation}\label{eq: 412}
(T_f g)\big|_{\mathbb C_I}= \langle f\star g , K_{\alpha}(\cdot,z) \rangle_{\alpha} = \langle  f\star g  , K_{\alpha}(\cdot,z) \rangle_{\alpha,I}=  \bigl(T^{(I)}_f g\bigr)\big|_{\mathbb C_I}. 
\end{equation}  
Since both sides are slice regular, the Extension Theorem yields $T_f g=T^{(I)}_f g$ on $\mathbb H$.  

\section{Quaternionic Fock--Carleson measures}\label{s:4} 
For a positive Borel measure \(\mu\) on \(\mathbb H\), we write
\[
\mu=\mu_{\mathbb R}+\mu^1,
\qquad
\mu_{\mathbb R}:=\mu|_{\mathbb R},
\qquad
\mu^1:=\mu|_{\mathbb H\setminus\mathbb R}.
\]
Thus \(\mu_{\mathbb R}\) is supported on \(\mathbb R\), while
\(\mu^1(\mathbb R)=0\). Throughout this section we assume, when needed, that
\(\mu^1\) is a Radon measure on \(\mathbb H\setminus\mathbb R\).

 Assume that \(\mu^1\) is Radon on \(\mathbb H\setminus\mathbb R\). Under the 
canonical identification
\[
\mathbb H\setminus\mathbb R\simeq \mathbb S\times\mathbb C^+,
\]
we may regard \(\mu^1\) as a Radon measure on \(\mathbb S\times\mathbb C^+\).

By the disintegration theorem for Radon measures, see \cite[10.4.16. Example]{MR2267655},  there is a measure $\nu_+$ on $\mathbb S$ and, for $\nu_+$-a.e. $I\in\mathbb S$, a Radon measure $\mu_I^+$ on
\[
\mathbb C_I^+=\{x+Iy\in\mathbb H:y>0\}
\]
such that
\[
\mu^1(E)=\int_{\mathbb S}\mu_I^+(E\cap \mathbb C_I^+)\,\mathrm d\nu_+(I),
\]
for every Borel set $E\subseteq \mathbb H\setminus\mathbb R$.
 
 Therefore, using the decomposition \(\mu=\mu_{\mathbb R}+\mu^1\), we obtain
\begin{equation}\label{misure1}
\int_{\mathbb{H}} f(x+Iy)\, \mathrm{d}\mu(x+Iy)
= \int_{\mathbb{R}} f(x)\, \mathrm{d}\mu_{\mathbb{R}}(x) 
+ \int_{\mathbb{S}} \int_{\mathbb{C}_I^+} f(x+Iy)\, \mathrm{d}\mu_I^+(x+Iy)\, \mathrm{d}\nu_+(I),
\end{equation}
for all $f \in L ^1(\mathbb{H},\mathrm{d}\mu)$.
 
\begin{definition}
    For $0<p<\infty$, a positive Borel measure $\mu$ is called a  quaternionic Fock--Carleson measure for $F^p_\alpha$  if there exists a constant $C>0$ such that for every $f \in F_\alpha^p$,
    \begin{equation}\label{9}
    \int_{\mathbb{H}} \big|f(w) \mathrm{e}^{-\frac{\alpha}{2}|w|^2}\big|^p\, \mathrm{d}\mu(w) 
    \leq C \|f\|_{p,\alpha}^p.
    \end{equation} 
\end{definition}
  
We next record a pointwise estimate analogous to the classical Fock-space
estimate and introduce the geometric testing sets used in the quaternionic
framework; compare also the Carleson boxes for Bergman spaces
\cite{MR374886} and the symmetric boxes in quaternionic Hardy and Bergman
spaces \cite{MR3664521}.

\begin{lemma}\label{lemma 4.1} 
Fix $I\in \mathbb{S}$. Let $\alpha, p, R > 0$. Then there exists a constant $C = C(p,\alpha,R) > 0$ such that
\[
\left| f(a) \mathrm{e}^{-\alpha |a|^2 / 2} \right|^p 
\leq \frac{C}{r^2} \int_{B_I(a,r)} \left| f(z) \mathrm{e}^{-\alpha |z|^2 / 2} \right|^p \, \mathrm{d}m_{2,I}(z),
\]
for all slice regular functions $f$, all $a \in \mathbb{C}_I$, and all $r \in (0,R]$, where $B_I(a,r)$ denotes the Euclidean disk centered at $a$ with radius $r$ on the complex plane $\mathbb{C}_I$.
\end{lemma}
 \begin{proof}
Fix $I,J\in\mathbb{S}$ with $J\perp I$. On $\mathbb{C}_I$ write the slice decomposition
\[
f_I(z)=F(z)+G(z)\,J,
\]
where $F,G:\mathbb{C}_I\to\mathbb{C}_I$ are holomorphic. Then pointwise
\[
\lvert f_I(z)\rvert^{2}=\lvert F(z)\rvert^{2}+\lvert G(z)\rvert^{2}
\quad\Longrightarrow\quad
\lvert f_I(z)\rvert^{p}\approx \lvert F(z)\rvert^{p}+\lvert G(z)\rvert^{p}.
\]
 
Applying \cite[Lemma 2.32]{MR2934601} to the holomorphic functions $F,G$  on $\mathbb{C}_I$, we obtain
\begin{equation*}
\begin{split}
\left| f(a) \mathrm{e}^{-\alpha |a|^2 / 2} \right|^p 
&\approx \lvert F(a)\mathrm{e}^{-\alpha |a|^2 / 2} \rvert^{p} + \lvert G(a)\mathrm{e}^{-\alpha |a|^2 / 2} \rvert^{p} \\
&\leq \frac{C}{r^2} \int_{B_I(a,r)} \left(\left| F(z) \mathrm{e}^{-\alpha |z|^2 / 2} \right|^p + \left| G(z) \mathrm{e}^{-\alpha |z|^2 / 2} \right|^p \right) \, \mathrm{d}m_{2,I} (z) \\
&\approx \frac{C}{r^2} \int_{B_I(a,r)} \left| f(z) \mathrm{e}^{-\alpha |z|^2 / 2} \right|^p \, \mathrm{d}m_{2,I}(z).
\end{split}
\end{equation*}
 \end{proof}

For \(z=x+yJ\in\mathbb H\), \(y\ge0\), and $I\in\mathbb S$,  set $
z_I:=x+yI\in\mathbb C_I  $. For \(r>0\), define the symmetric box centered at \(z\) by
\[
S(z,r):=
\bigcup_{I\in\mathbb S} B_I(z_I,r),
\]
where \(B_I(z_I,r)\) denotes the Euclidean disk in the slice \(\mathbb C_I\)
with center \(z_I\) and radius \(r\).
Equivalently, \(S(z,r)\) is the axially symmetric completion of the disk
\(B_J(z,r)\subset\mathbb C_J\).
\begin{theorem}\label{thm:Fock--Carleson}
Let \(\alpha>0\), \(0<p<\infty\), and \(0<r<\infty\). Let \(\mu\) be a positive
Borel measure on \(\mathbb H\) such that $
\mu^1:=\mu|_{\mathbb H\setminus\mathbb R}$ 
is a Radon measure on \(\mathbb H\setminus\mathbb R\). Then the following are
equivalent:
\begin{enumerate}
    \item[(a)] \(\mu\) is a quaternionic Fock--Carleson measure for \(F_\alpha^p\).
    
    \item[(b)] There exists a constant \(C > 0\) such that
    \[
    \int_{\mathbb H}
    \left| k_z(w)\, \mathrm{e}^{-\frac{\alpha}{2}|w|^2} \right|^p
    \mathrm{d}\mu(w) \le C,
    \]
    for all \(z \in \mathbb H\), where
    \(k_z(w)=\mathrm{e}_*^{\alpha w\overline z}/\mathrm{e}^{\alpha |z|^2/2}\)
    is the normalized reproducing kernel.
    
    \item[(c)] There exists a constant \(C > 0\) such that
    \[
    \mu(S(z,r)) \le C
    \qquad \text{for all } z\in\mathbb H.
    \]
\end{enumerate}
\end{theorem}
\begin{proof}
We prove \textup{(a)}\(\Rightarrow\)\textup{(b)},
\textup{(c)}\(\Rightarrow\)\textup{(a)}, and
\textup{(b)}\(\Rightarrow\)\textup{(c)}.

Taking \(f=k_z\) in the Carleson embedding inequality and using
Theorem~\ref{thm dengju}, condition \textup{(a)} immediately implies
\textup{(b)}.

We next prove \textup{(c)}\(\Rightarrow\)\textup{(a)}. Decompose $
\mu=\mu_{\mathbb R}+\mu^1$. We treat the two parts separately.

\noindent\textbf{Step 1: The part supported on \(\mathbb R\).}
Since \(\mu_{\mathbb R}\) is supported on \(\mathbb R\), for every
\(I\in\mathbb S\) and every \(z\in\mathbb C_I\),
\[
S(z,r)\cap\mathbb R=B_I(z,r)\cap\mathbb R .
\]
Hence
\[
\mu_{\mathbb R}(B_I(z,r))
=
\mu_{\mathbb R}(S(z,r))
\le C ,
\qquad z\in\mathbb C_I .
\]
By the classical Fock--Carleson characterization on the complex Fock space
\(\mathcal F_\alpha^p(\mathbb C_I)\), the measure \(\mu_{\mathbb R}\), regarded
as a measure on \(\mathbb C_I\), is a Fock--Carleson measure; see
\cite[Theorem~3.29]{MR2934601}.

Let \(f\in F_\alpha^p\). By the Splitting Lemma, on \(\mathbb C_I\) we may write
\[
f_I=F+GJ,\qquad J\perp I,
\]
where \(F,G\in\mathcal F_\alpha^p(\mathbb C_I)\). Therefore
\[
\begin{aligned}
\int_{\mathbb H}
\left|f(z)e^{-\frac{\alpha}{2}|z|^2}\right|^p\,d\mu_{\mathbb R}(z)
&=
\int_{\mathbb R}
\bigl(|F(z)|^2+|G(z)|^2\bigr)^{p/2}
e^{-\frac{p\alpha}{2}|z|^2}\,d\mu_{\mathbb R}(z)  \\
&\lesssim
\int_{\mathbb C_I}
\left|F(z)e^{-\frac{\alpha}{2}|z|^2}\right|^p\,d\mu_{\mathbb R}(z)  +
\int_{\mathbb C_I}
\left|G(z)e^{-\frac{\alpha}{2}|z|^2}\right|^p\,d\mu_{\mathbb R}(z)  \\
&\lesssim
\|F\|_{p,\alpha,I}^p+\|G\|_{p,\alpha,I}^p  \\
&\lesssim
\|f\|_{p,\alpha,I}^p
\approx
\|f\|_{p,\alpha}^p .
\end{aligned}
\]
Thus \(\mu_{\mathbb R}\) is a quaternionic Fock--Carleson measure.

\noindent\textbf{Step 2: The part supported on \(\mathbb H\setminus\mathbb R\).}
Let \(0<\rho<r/4\). Choose a \(\rho\)-lattice
\(\{z_n=x_n+iy_n\}_{n\ge1}\subset\mathbb C^+\), with \(y_n>0\), such that
\[
\mathbb C^+\subset\bigcup_{n\ge1}D(z_n,\rho)
\]
and the enlarged disks \(D(z_n,2\rho)\) have bounded overlap. For each
\(I\in\mathbb S\), set
\[
(z_n)_I:=x_n+y_nI\in\mathbb C_I^+ .
\]
Then the family \(\{B_I((z_n)_I,\rho)\}_{n\ge1}\) covers
\(\mathbb C_I^+\), and the family
\(\{B_I((z_n)_I,2\rho)\}_{n\ge1}\) has bounded overlap, with constants
independent of \(I\).

For \(f\in F_\alpha^p\), set
\[
\mathcal I(f)
:=
\int_{\mathbb H}
\left|f(w)e^{-\frac{\alpha}{2}|w|^2}\right|^p\,d\mu^1(w).
\]
Using the disintegration formula \eqref{misure1}, we obtain
\[
\mathcal I(f)
=
\int_{\mathbb S}\int_{\mathbb C_I^+}
\left|f(w)e^{-\frac{\alpha}{2}|w|^2}\right|^p
\,d\mu_I^+(w)\,d\nu_+(I).
\]
Since \(\{B_I((z_n)_I,\rho)\}_n\) covers \(\mathbb C_I^+\), it follows that
\[
\mathcal I(f)
\le
\int_{\mathbb S}\sum_n
\int_{B_I((z_n)_I,\rho)}
\left|f(w)e^{-\frac{\alpha}{2}|w|^2}\right|^p
\,d\mu_I^+(w)\,d\nu_+(I).
\]
By Lemma~\ref{lemma 4.1}, for \(w\in B_I((z_n)_I,\rho)\),
\[
\left|f(w)e^{-\frac{\alpha}{2}|w|^2}\right|^p
\lesssim
\int_{B_I((z_n)_I,2\rho)}
\left|f(u)e^{-\frac{\alpha}{2}|u|^2}\right|^p
\,dm_{2,I}(u),
\]
where the implicit constant depends only on \(p,\alpha\), and \(\rho\). Hence
\[
\begin{aligned}
\mathcal I(f)
&\lesssim
\int_{\mathbb S}\sum_n
\mu_I^+\bigl(B_I((z_n)_I,\rho)\bigr)
\int_{B_I((z_n)_I,2\rho)}
\left|f(u)e^{-\frac{\alpha}{2}|u|^2}\right|^p
\,dm_{2,I}(u)\,d\nu_+(I).
\end{aligned}
\]

Fix \(J\in\mathbb S\). If \(u=x+yI\in B_I((z_n)_I,2\rho)\), then the
corresponding point \(x+yJ\in\mathbb C_J\) belongs to
\(B_J((z_n)_J,2\rho)\). By the representation formula,
\[
|f(x+yI)|^p
\lesssim
|f(x+yJ)|^p+|f(x-yJ)|^p .
\]
Consequently,
\[
\begin{aligned}
&\int_{B_I((z_n)_I,2\rho)}
\left|f(u)e^{-\frac{\alpha}{2}|u|^2}\right|^p
\,dm_{2,I}(u)   \lesssim
\int_{B_J((z_n)_J,2\rho)}
\bigl(|f(w)|^p+|f(\bar w)|^p\bigr)
e^{-\frac{p\alpha}{2}|w|^2}\,dm_{2,J}(w).
\end{aligned}
\]
Set
\[
A_n^J
:=
\int_{B_J((z_n)_J,2\rho)}
\bigl(|f(w)|^p+|f(\bar w)|^p\bigr)
e^{-\frac{p\alpha}{2}|w|^2}\,dm_{2,J}(w).
\]
Since the lattice \(\{z_n\}\subset\mathbb C^+\) was chosen independently of
\(I\), the quantity \(A_n^J\) is independent of \(I\). Thus
\[
\mathcal I(f)
\lesssim
\sum_n A_n^J
\int_{\mathbb S}
\mu_I^+\bigl(B_I((z_n)_I,\rho)\bigr)\,d\nu_+(I).
\]

For each \(n\),
\[
\bigcup_{I\in\mathbb S}
\bigl(B_I((z_n)_I,\rho)\cap\mathbb C_I^+\bigr)
\subset S((z_n)_J,\rho)
\subset S((z_n)_J,r).
\]
Therefore, by the disintegration formula and condition \textup{(c)},
\[
\begin{aligned}
\int_{\mathbb S}
\mu_I^+\bigl(B_I((z_n)_I,\rho)\bigr)\,d\nu_+(I)
&=
\mu^1\!\left(
\bigcup_{I\in\mathbb S}
\bigl(B_I((z_n)_I,\rho)\cap\mathbb C_I^+\bigr)
\right) \\
&\le
\mu(S((z_n)_J,r))
\le C .
\end{aligned}
\]
It follows that
\[
\mathcal I(f)
\lesssim
\sum_n A_n^J .
\]
Since the family \(\{B_J((z_n)_J,2\rho)\}_{n\ge1}\) has bounded overlap,
\[
\begin{aligned}
\sum_n A_n^J
&\lesssim
\int_{\mathbb C_J}
\bigl(|f(w)|^p+|f(\bar w)|^p\bigr)
e^{-\frac{p\alpha}{2}|w|^2}\,dm_{2,J}(w)  \\
&\lesssim
\|f\|_{p,\alpha,J}^p
\approx
\|f\|_{p,\alpha}^p .
\end{aligned}
\]
Hence
\[
\mathcal I(f)\lesssim \|f\|_{p,\alpha}^p,
\]
and therefore \(\mu^1\) is a quaternionic Fock--Carleson measure. Combining
this with Step 1 proves \textup{(c)}\(\Rightarrow\)\textup{(a)}.

It remains to prove \textup{(b)}\(\Rightarrow\)\textup{(c)}. Fix
\(z=x+yI\in\mathbb H\) with \(y\ge0\). For each \(J\in\mathbb S\) and
\(w\in\mathbb C_J^+\), the representation formula applied to the \(z\)-variable
of the kernel gives
\[
|k_{z_J}(w)|
\le
|k_z(w)|+|k_{\bar z}(w)|,
\]
where \(z_J:=x+yJ\). Hence
\[
|k_{z_J}(w)|^p
\lesssim
|k_z(w)|^p+|k_{\bar z}(w)|^p .
\]
Multiplying by \(e^{-\frac{p\alpha}{2}|w|^2}\), integrating with respect to
\(d\mu_J^+(w)\,d\nu_+(J)\), and using condition \textup{(b)}, we get
\[
\int_{\mathbb S}\int_{\mathbb C_J^+}
|k_{z_J}(w)e^{-\frac{\alpha}{2}|w|^2}|^p
\,d\mu_J^+(w)\,d\nu_+(J)
\le C .
\]
Since \(z_J,w\in\mathbb C_J\), the classical exponential identity gives
\[
|k_{z_J}(w)e^{-\frac{\alpha}{2}|w|^2}|^p
=
e^{-\frac{p\alpha}{2}|z_J-w|^2}.
\]
Thus
\[
\int_{\mathbb S}\int_{\mathbb C_J^+}
e^{-\frac{p\alpha}{2}|z_J-w|^2}
\,d\mu_J^+(w)\,d\nu_+(J)
\le C .
\]
On the other hand,
\[
S(z,r)\cap\mathbb C_J^+
=
B_J(z_J,r)\cap\mathbb C_J^+,
\]
and hence, for \(w\in S(z,r)\cap\mathbb C_J^+\),
\[
e^{-\frac{p\alpha}{2}|z_J-w|^2}
\ge
e^{-\frac{p\alpha}{2}r^2}.
\]
Therefore
\[
e^{-\frac{p\alpha}{2}r^2}\mu^1(S(z,r))
\le
\int_{\mathbb S}\int_{\mathbb C_J^+}
e^{-\frac{p\alpha}{2}|z_J-w|^2}
\,d\mu_J^+(w)\,d\nu_+(J)
\le C .
\]
This gives a uniform bound for \(\mu^1(S(z,r))\).

It remains to estimate the real part. Since \(\mu_{\mathbb R}\) is supported on
\(\mathbb R\),
\[
\mu_{\mathbb R}(S(z,r))
=
\mu_{\mathbb R}(B_I(z,r)).
\]
For \(w\in B_I(z,r)\cap\mathbb R\), we have \(z,w\in\mathbb C_I\) and
\[
|k_z(w)e^{-\frac{\alpha}{2}|w|^2}|^p
=
e^{-\frac{p\alpha}{2}|z-w|^2}
\ge
e^{-\frac{p\alpha}{2}r^2}.
\]
Hence, using condition \textup{(b)} again,
\[
\begin{aligned}
e^{-\frac{p\alpha}{2}r^2}\mu_{\mathbb R}(S(z,r))
&\le
\int_{B_I(z,r)}
|k_z(w)e^{-\frac{\alpha}{2}|w|^2}|^p\,d\mu_{\mathbb R}(w) \\
&\le
\int_{\mathbb H}
|k_z(w)e^{-\frac{\alpha}{2}|w|^2}|^p\,d\mu(w)
\le C .
\end{aligned}
\]
Combining the estimates for \(\mu^1\) and \(\mu_{\mathbb R}\), we obtain
\[
\mu(S(z,r))\le C,
\qquad z\in\mathbb H.
\]
Thus \textup{(c)} holds. The proof is complete.
\end{proof}

\begin{remark}
Condition \textup{(c)} is independent of $p$ and $\alpha$. Consequently, if condition \textup{(a)} holds for one choice of $p>0$ and $\alpha>0$, then it holds for every such choice, with constants depending on $p$ and $\alpha$. 
Likewise, condition \textup{(a)} is independent of the radius $r$.  If condition \textup{(c)} holds for one $r>0$, then it holds for every $r>0$, with constants depending on $r$.
\end{remark}

\begin{definition}
For $0<p<\infty$, a positive Borel measure \(\mu\) on \(\mathbb H\) is called a
vanishing quaternionic Fock--Carleson measure for \(F_\alpha^p\) if
\[
\lim_{n\to\infty}
\int_{\mathbb H}
\bigl|f_n(z)\mathrm e^{-\frac{\alpha}{2}|z|^2}\bigr|^p\,\mathrm d\mu(z)=0,
\]
whenever \(\{f_n\}\) is a bounded sequence in \(F_\alpha^p\) such that
\(f_n\to0\) uniformly on compact subsets of \(\mathbb H\).
\end{definition}

\begin{remark}\label{equivalent}
For slice regular functions, compact convergence on \(\mathbb H\) is equivalent
to compact convergence on any fixed slice. Indeed, fix \(I_0\in\mathbb S\). By
the representation formula, for \(q=x+yJ\in\mathbb H\),
\[
	f_n(q)
=
\frac12(1-JI_0)f_n(x+yI_0)
+
\frac12(1+JI_0)f_n(x-yI_0).
\]
Hence, if \(K\subset\mathbb H\) is compact, set $
K_{I_0}
:=
\{\operatorname{Re}q\pm |\operatorname{Im}q|I_0:\ q\in K\}$. 
The maps
\[
q\longmapsto \operatorname{Re}q\pm |\operatorname{Im}q|I_0
\]
are continuous from \(\mathbb H\) into \(\mathbb C_{I_0}\). Thus
\(K_{I_0}\) is the union of two compact images of \(K\), and hence is compact in
\(\mathbb C_{I_0}\). Moreover, by the representation formula, we have
\[
\sup_{q\in K}|f_n(q)|
\le C\sup_{\zeta\in K_{I_0} }|f_n(\zeta)|.
\]
Thus uniform convergence on compact subsets of one slice implies uniform
convergence on compact subsets of \(\mathbb H\). The converse is immediate.
\end{remark}
 \begin{theorem}\label{thm:vanishing-Carleson}
Let \(\alpha>0\), \(0<p<\infty\), and \(0<r<\infty\). Let \(\mu\) be a positive
Borel measure on \(\mathbb H\) such that $
\mu^1 $ 
is a Radon measure on \(\mathbb H\setminus\mathbb R\). Then the following are equivalent:
\begin{itemize}
\item[(a)] \(\mu\) is a vanishing quaternionic Fock--Carleson measure for \(F_\alpha^p\).
\item[(b)] For every \(I\in \mathbb{S}\),
\[
\int_{\mathbb{H}}
\bigl|k_{z}(w)\, \mathrm{e}^{-\frac{\alpha}{2}|w|^{2}}\bigr|^{p}
\, \mathrm{d}\mu(w) \longrightarrow 0
\qquad \text{as } |z|\to\infty,\ z\in \mathbb{C}_I .
\]
\item[(c)] For every \(I\in \mathbb{S}\), we have
\[
\mu\bigl(S(z,r)\bigr)\to 0
\qquad \text{as } |z|\to\infty,\ z\in \mathbb C_I .
\]
\end{itemize}
\end{theorem}

\begin{proof}

We prove \textup{(b)}\(\Rightarrow\)\textup{(c)}, \textup{(a)}\(\Rightarrow\)\textup{(b)}, and \textup{(c)}\(\Rightarrow\)\textup{(a)}.

The implication \textup{(b)}\(\Rightarrow\)\textup{(c)} follows from the estimate established in the proof of Theorem~\ref{thm:Fock--Carleson}: for all \(z\in\mathbb H\),
\[
\mu\bigl(S(z,r)\bigr)
\lesssim
\int_{\mathbb H}\bigl|k_z(w) \mathrm e^{-\frac{\alpha}{2}|w|^2}\bigr|^p\,\mathrm d\mu(w)
+
\int_{\mathbb H}\bigl|k_{\bar z}(w) \mathrm e^{-\frac{\alpha}{2}|w|^2}\bigr|^p\,\mathrm d\mu(w).
\]
Hence, if \textup{(b)} holds, then \(\mu(S(z,r))\to 0\) as \(|z|\to\infty\) on every slice, since \(|\bar z|=|z|\).

We next prove \textup{(a)}\(\Rightarrow\)\textup{(b)}. Let \(I\in\mathbb S\), let
\(\{\zeta_n\}\subset \mathbb C_I\) satisfy \(|\zeta_n|\to\infty\), and set
\[
f_n(w):=k_{\zeta_n}(w),\qquad w\in\mathbb H.
\]
Then \(\|f_n\|_{p,\alpha}\approx 1\) for all \(n\).  

Moreover, \(k_{\zeta_n}\to0\) uniformly on compact subsets of \(\mathbb C_I\).    By the Remark \ref{equivalent}, this implies that
\(k_{\zeta_n}\to0\) uniformly on compact subsets of \(\mathbb H\). Hence, by the
vanishing Fock--Carleson property,
\[
\int_{\mathbb H}\bigl|k_{\zeta_n}(w) \mathrm e^{-\frac{\alpha}{2}|w|^2}\bigr|^p\,\mathrm d\mu(w)\to 0.
\]
Since the sequence \(\{\zeta_n\}\) was arbitrary, this proves \textup{(b)}.

It remains to prove \textup{(c)}\(\Rightarrow\)\textup{(a)}. Let \(\{f_n\}\) be a
bounded sequence in \(F_\alpha^p\) such that \(f_n\to0\) uniformly on compact
subsets of \(\mathbb H\). We prove that
\[
\int_{\mathbb H}
\bigl|f_n(w)\mathrm e^{-\frac{\alpha}{2}|w|^2}\bigr|^p\,\mathrm d\mu(w)
\to0.
\]
Decompose \(\mu=\mu_{\mathbb R}+\mu^1\), where
\(\operatorname{supp}\mu_{\mathbb R}\subset\mathbb R\) and \(\mu^1(\mathbb R)=0\).

\noindent\textbf{The part supported on \(\mathbb R\).}
Fix \(I\in\mathbb S\) and choose \(J\in\mathbb S\) with \(J\perp I\). Since
\[
S(z,r)\cap\mathbb R=B_I(z,r)\cap\mathbb R,\qquad z\in\mathbb C_I,
\]
condition \textup{(c)} gives
\[
\mu_{\mathbb R}\bigl(B_I(z,r)\bigr)\to0
\qquad\text{as } |z|\to\infty,\ z\in\mathbb C_I.
\]
By the Splitting Lemma, write on \(\mathbb C_I\)
\[
f_n(z)=F_n(z)+G_n(z)J,\qquad J\in\mathbb{S},\ J\perp I,
\]
where \(F_n,G_n\in \mathcal F_\alpha^p(\mathbb C_I)\). Moreover,
\[
|f_n(z)|^2=|F_n(z)|^2+|G_n(z)|^2,\qquad z\in\mathbb C_I.
\]
Hence \(\{F_n\}\) and \(\{G_n\}\) are bounded in
\(\mathcal F_\alpha^p(\mathbb C_I)\), and both converge to \(0\) uniformly on
compact subsets of \(\mathbb C_I\). By the classical vanishing Fock--Carleson
theorem on \(\mathbb C_I\), applied to the measure \(\mu_{\mathbb R}\), we get
\[
\int_{\mathbb R}
\bigl|F_n(z)\mathrm e^{-\frac{\alpha}{2}|z|^2}\bigr|^p\,\mathrm d\mu_{\mathbb R}(z)
\to0,
\qquad
\int_{\mathbb R}
\bigl|G_n(z)\mathrm e^{-\frac{\alpha}{2}|z|^2}\bigr|^p\,\mathrm d\mu_{\mathbb R}(z)
\to0;
\]
see \cite[Theorem 3.30]{MR2934601}. Therefore
\[
\begin{aligned}
\int_{\mathbb{R}} \bigl|f_n(z)\,\mathrm{e}^{-\frac{\alpha}{2}|z|^2}\bigr|^{p}\,\mathrm{d}\mu_{\mathbb{R}}(z)
&= \int_{\mathbb{R}} \Bigl(|F_n(z)|^2+|G_n(z)|^2\Bigr)^{p/2} \mathrm{e}^{-\frac{p\alpha}{2}|z|^2}\, \mathrm{d}\mu_{\mathbb{R}}(z) \\
&\lesssim \!\left(\,
\int_{\mathbb{R}} \bigl|F_n(z)\,\mathrm{e}^{-\frac{\alpha}{2}|z|^2}\bigr|^{p}\,\mathrm{d}\mu_{\mathbb{R}}(z)
+\int_{\mathbb{R}} \bigl|G_n(z)\,\mathrm{e}^{-\frac{\alpha}{2}|z|^2}\bigr|^{p}\,\mathrm{d}\mu_{\mathbb{R}}(z)
\right)\to0.
\end{aligned}
\]

\noindent\textbf{The part supported on \(\mathbb H\setminus\mathbb R\).}
We first note that condition \textup{(c)} implies the bounded box condition. Fix
\(J\in\mathbb S\). Since \(S(x+yI,r)=S(x+yJ,r)\) for \(y\ge0\), it is enough to
consider centers in \(\mathbb C_J\). The limit in \textup{(c)} gives boundedness
outside a large disk in \(\mathbb C_J\), while on the remaining compact set it
follows from local finiteness. Thus there is \(M>0\) such that
\[
\mu(S(z,r))\le M,\qquad z\in\mathbb H.
\]

Choose an \(r\)-lattice \(\{z_k\}_{k\ge1}\subset\mathbb C_J^+\), ordered so that
\(|z_k|\to\infty\), whose disks \(B_J(z_k,r)\) cover \(\mathbb C_J^+\) and whose
enlarged disks \(B_J(z_k,2r)\) have bounded overlap. By the lattice estimate
obtained in the proof of Theorem~\ref{thm:Fock--Carleson}, there is a constant
independent of \(n\) such that
\[
\mathcal{I}  (f_n) :=\int_{\mathbb{H}} \bigl|f_n(w)\,\mathrm{e}^{-\alpha|w|^2/2}\bigr|^{p}\,\mathrm{d}\mu^{1}(w)
\le C\sum_k \mu \bigl(S(z_k,r)\bigr) A_{n,k}^J,
\]
where
\[
A_{n,k}^J
:=
\int_{B_J(z_k,2r)}
\bigl(|f_n(w)|^p+|f_n(\bar w)|^p\bigr)
e^{-p\alpha |w|^2/2}\,dm_{2,J}(w).
\]
Given \(\varepsilon>0\), condition \textup{(c)} and \(|z_k|\to\infty\) give
\(N\) such that
\[
\mu(S(z_k,r))<\varepsilon,\qquad k>N.
\]
Therefore
\[
\begin{aligned}
\mathcal I(f_n)
&\le C M\sum_{k=1}^{N} A_{n,k}^J
   + C\varepsilon\sum_{k=N+1}^{\infty} A_{n,k}^J .
\end{aligned}
\]
Since \(N\) is fixed, there exists \(R>0\) such that
\[
\bigcup_{k=1}^{N} B_J(z_k,2r)
\subset \{w\in\mathbb C_J: |w|\le R\}.
\]
Moreover, the conjugates of these disks are also contained in the same compact
disk \(\{w\in\mathbb C_J: |w|\le R\}\). Since \(f_n\to0\) uniformly on compact
subsets of \(\mathbb H\), we have $
\sum_{k=1}^{N}A_{n,k}^J\to0$ as $n\to\infty$. 

On the other hand, bounded overlap and the change of variables \(w\mapsto\bar w\)
give
\[
\begin{aligned}
\sum_{k=1}^{\infty} A_{n,k}^J
&\lesssim
\int_{\mathbb C_J}
\bigl(|f_n(w)|^p+|f_n(\bar w)|^p\bigr)
\mathrm e^{-\frac{p\alpha}{2}|w|^2}\,\mathrm dm_{2,J}(w)  \\
&\lesssim
\|f_n\|_{p,\alpha,J}^p
\lesssim
\|f_n\|_{p,\alpha}^p .
\end{aligned}
\]
The last quantity is uniformly bounded in \(n\). Hence
\[
\limsup_{n\to\infty}\mathcal I(f_n)\lesssim \varepsilon.
\]
Since \(\varepsilon>0\) is arbitrary, \(\mathcal I(f_n)\to0\). Combining this
with the estimate for \(\mu_{\mathbb R}\), we obtain
\[
\int_{\mathbb H}
\bigl|f_n(w)\mathrm e^{-\frac{\alpha}{2}|w|^2}\bigr|^p\,\mathrm d\mu(w)
\to0.
\]
This proves \textup{(c)}\(\Rightarrow\)\textup{(a)}, and hence completes the proof.
\end{proof}

\begin{remark}
The preceding theorem shows that the notion of vanishing quaternionic
Fock--Carleson measure is independent of the particular choice of \(p>0\) and
\(\alpha>0\). It is also independent of the radius \(r>0\) appearing in
condition \textup{(c)}.
\end{remark}

\begin{corollary}
Let \(\mu\) be a positive Borel measure on \(\mathbb H\) such that
\(\mu^1\) is a Radon measure. Fix \(I\in\mathbb S\)
and \(r>0\), and let \(\{z_k\}_{k\ge1}\subset\mathbb C_I\) be an \(r\)-lattice in
\(\mathbb C_I\). Then the following hold:
\begin{enumerate}[(a)]
\item \(\mu\) is a quaternionic Fock--Carleson measure if and only if
\[
\sup_{k\ge1}\mu(S(z_k,r))<\infty .
\]

\item \(\mu\) is a vanishing quaternionic Fock--Carleson measure if and only if
\[
\mu(S(z_k,r))\to0
\qquad\text{as } k\to\infty .
\]
\end{enumerate}
Both characterizations are independent of the particular choice of the slice
\(I\), the radius \(r\), and the \(r\)-lattice.
\end{corollary}

\section{The Berezin transform of slice functions and BMO}\label{sec:berezin-bmo} 

\begin{definition}
A slice function $f:\mathbb H\to\mathbb H$ is said to be slice continuous if, for every
$I\in\mathbb S$, its restriction $
f_I $ is continuous on the slice $\mathbb C_I$. We write $C(\mathbb H)$ for the space of slice continuous functions on $\mathbb H$.

We also define
\[
C_0(\mathbb{H})
:= \Bigl\{\, f\in C(\mathbb{H}) : \text{for every } I\in \mathbb{S},\ f_I \ \text{vanishes at infinity on } \mathbb{C}_I \Bigr\}.
\]
 
\end{definition}
Thus \(C_0(\mathbb H)\) is defined slice-wise: the vanishing condition is imposed
on each complex slice \(\mathbb C_I\), rather than with respect to the Euclidean
topology of \(\mathbb H\).  For \(1\le p<\infty\), \(L_s^p(\mathbb H,\mathrm dV)\) denotes the space of
measurable slice functions \(h\) such that
\[
\|h\|_{L_s^p(\mathrm dV)}^p
:=
\int_{\mathbb H}|h(z)|^p\,\mathrm dV(z)<\infty,
\]
and   \(L_s^\infty(\mathbb H,\mathrm dV)\) denotes the space of essentially bounded measurable slice functions on \(\mathbb H\) with respect to \(\mathrm dV\).

For a measurable function \(\varphi:\mathbb H\to\mathbb H\), we say that \(\varphi\) satisfies condition \((I_p)\) if
\[
\int_{\mathbb C_I}|K_\alpha(z,a)|^2\,|\varphi(z)|^p\,\mathrm d\lambda_{\alpha,I}(z)<\infty
\qquad\text{for all } a\in\mathbb C_I,\ I\in\mathbb S.
\]
On each slice \(\mathbb C_I\), the kernel has the exponential form
\[
|K_\alpha(z,a)|=\exp\bigl(\alpha\operatorname{Re}(z\bar a)\bigr),\qquad\text{for all } a\in\mathbb C_I. 
\]
Hence \((I_p)\) is equivalently characterized by
\[
\int_{\mathbb C_I}|K_\alpha(z,a)|\,|\varphi(z)|^p\,\mathrm d\lambda_{\alpha,I}(z)<\infty
\qquad\text{for all } a\in\mathbb C_I,\ I\in\mathbb S.
\]
In particular, every bounded measurable function on $\mathbb{H}$ satisfies \((I_p)\).

\begin{lemma}\label{lem:Ip-one-slice}
Let \(\varphi\) be a slice function on \(\mathbb H\). Then condition \((I_p)\) can be verified on a single slice: if there exists \(I_0\in\mathbb S\) such that
\[
\int_{\mathbb C_{I_0}} |K_\alpha(z,a)|^2\,|\varphi(z)|^p\,\mathrm d\lambda_{\alpha,I_0}(z)<\infty
\qquad\text{for all } a\in\mathbb C_{I_0},
\]
then the same condition holds on every slice \(\mathbb C_J\), \(J\in\mathbb S\).
\end{lemma}
\begin{proof}
  
Fix \(J\in\mathbb S\) and \(b=u+Jv\in\mathbb C_J\). Set \(a:=u+I_0v\in\mathbb C_{I_0}\).
For \(z=x+Jy\), the representation formula for slice functions gives
\[
|\varphi(x+Jy)|^p\lesssim \Bigl(|\varphi(x+I_0y)|^p+|\varphi(x-I_0y)|^p\Bigr).
\]
On the other hand, by the exponential form of the kernel on each slice,
\[
|K_\alpha(x+Jy,b)|^2
=\exp\bigl(2\alpha(xu+yv)\bigr)
=|K_\alpha(x+I_0y,a)|^2
=|K_\alpha(x-I_0y,\bar a)|^2.
\]
Then, by a change of variables,
\begin{equation}
\begin{split}
\int_{\mathbb{C}_J} |K_\alpha(z,b)|^2 |\varphi(z)|^p \,\mathrm d\lambda_{\alpha,J}(z)
&\lesssim \int_{\mathbb{C}_{I_0}} |K_\alpha(\zeta,a)|^2 |\varphi(\zeta)|^p \,\mathrm d\lambda_{\alpha,I_0}(\zeta) \\
&\quad + \int_{\mathbb{C}_{I_0}} |K_\alpha(\zeta,\overline a)|^2 |\varphi(\zeta)|^p \,\mathrm d\lambda_{\alpha,I_0}(\zeta),
\end{split}
\nonumber
\end{equation}
and the right-hand side is finite by assumption. Consequently, for slice functions, one may
check $(I_p)$ on a single slice $\mathbb C_{I_0}$.
   
\end{proof}

Let \(f\) be a slice function satisfying \((I_1)\). For each \(I\in\mathbb S\), we define on \(\mathbb C_I\) the slice-wise Berezin transform by
\[
(B_\alpha f)(z):=\int_{\mathbb C_I}|k_z(w)|^2\,f(w)\,\mathrm d\lambda_{\alpha,I}(w),
\qquad z\in\mathbb C_I.
\] 
The next theorem shows that this slice-wise definition gives rise to a slice function on \(\mathbb H\).
 
\begin{theorem}\label{thm:Berezin-slice}
Let \(f\) be a slice function on \(\mathbb H\) satisfying \((I_1)\). Then the slice-wise formula
\[
(B_\alpha f)(z)
:=\int_{\mathbb C_I}|k_z(w)|^2\,f(w)\,\mathrm d\lambda_{\alpha,I}(w)
=\frac{\alpha}{\pi}\int_{\mathbb C_I} \mathrm e^{-\alpha|w-z|^2}f(w)\,\mathrm d m_{2,I}(w),
\qquad z\in\mathbb C_I,
\]
defines a slice function \(B_\alpha f\) on \(\mathbb H\). Moreover, if on \(\mathbb C_I\) one writes
\[
f_I=F+GJ,\qquad J\perp I,
\]
then
\[
(B_\alpha f)(z)=B_{\alpha,I}F(z)+B_{\alpha,I}G(z)\,J,\qquad z\in\mathbb C_I,
\]
where \(B_{\alpha,I}\) denotes the classical Berezin transform on the complex Fock space over \(\mathbb C_I\); see  \cite{MR2934601}.
In particular, if \(f\) is intrinsic, then \(B_\alpha f\) is intrinsic.
\end{theorem} 

\begin{proof}
Write points as $z=x+yI\in\mathbb{C}_I$ and $w=x'+y'I\in\mathbb{C}_I$ with $x,x',y,y'\in\mathbb{R}$.
Since $f$ is a slice function, there exist stem functions $A,B:\mathbb{R}^2\to\mathbb H$ with
\[
A(x,-y)=A(x,y),\qquad B(x,-y)=-B(x,y),
\]
such that for every $I\in\mathbb S$,
\[
f(x'+y'I)=A(x',y')+I\,B(x',y').
\]
Using the explicit form of the kernel and measure, for $z=x+yI\in\mathbb{C}_I$, we compute
\[
B_\alpha f(x+yI)
=\frac{\alpha}{\pi}\int_{\mathbb{R}^2} \mathrm{e}^{-\alpha\big((x'-x)^2+(y'-y)^2\big)}\big(A(x',y')+I\,B(x',y')\big)\,dx'\,dy'.
\]

Define the functions
\[
a(x,y):=\frac{\alpha}{\pi}\int_{\mathbb{R}^2} \mathrm{e}^{-\alpha\big((x'-x)^2+(y'-y)^2\big)}A(x',y')\,dx'\,dy',
\]
\[
b(x,y):=\frac{\alpha}{\pi}\int_{\mathbb{R}^2} \mathrm{e}^{-\alpha\big((x'-x)^2+(y'-y)^2\big)}B(x',y')\,dx'\,dy'.
\]
Then
\[
B_\alpha f(x+yI)=a(x,y)+I\,b(x,y)\qquad(I\in\mathbb S).
\]
By variable substitution, the parity of $A, B$ is preserved:
\[
a(x,-y)=a(x,y),\qquad b(x,-y)=-b(x,y).
\]
In particular, if \(z=x\in\mathbb R\), then \(b(x,0)=0\), so \(B_\alpha f(x)\) is independent of the choice of \(I\).
Consequently, by the standard characterization of slice functions via stem functions, the map
\[
x+yI\longmapsto B_\alpha f(x+yI):=a(x,y)+I\,b(x,y)
\]
is a slice function on \(\mathbb H\).

If, in addition, $f$ is intrinsic, then the stem components take values in $\mathbb{R}$ (or, more precisely, $A(x,y),B(x,y)\in\mathbb{R}$ for all $(x,y)$). Hence $a(x,y),b(x,y)\in\mathbb{R}$, i.e., $B_\alpha f$ is intrinsic. 

On $\mathbb C_I$ write $f(w)=F(w)+G(w)J$ with $J\perp I$.
Since the weight $|k_{z }(w)|^2 $ is a real scalar kernel,
it commutes with $J$ and we have
\begin{equation}
\begin{split}
(B_\alpha f)(z)
&= \int_{\mathbb C_I} |k_{z}(w)|^2 F(w) \,\mathrm d\lambda_{\alpha,I}(w)  + \left( \int_{\mathbb C_I} |k_{z}(w)|^2 G(w) \,\mathrm d\lambda_{\alpha,I}(w) \right) J \\
&= B_{\alpha,I} F(z) + B_{\alpha,I} G(z) \, J.
\end{split}
\nonumber
\end{equation}
\end{proof}

 \begin{theorem}\label{thm 5.9}
Suppose \(\alpha,\beta>0\) and \(f\) is a real-valued nonnegative slice function
on \(\mathbb H\) satisfying \((I_1)\). For \(1\le p\le \infty\), the following
hold:
\begin{enumerate}[(a)]
\item \(B_{\alpha} f \in L_s^p(\mathbb H,\mathrm dV)\) if and only if
\(B_{\beta} f \in L_s^p(\mathbb H,\mathrm dV)\).
\item \(B_{\alpha} f \in C_0(\mathbb H)\) if and only if
\(B_{\beta} f \in C_0(\mathbb H)\).
\end{enumerate}
\end{theorem}

\begin{proof}
Since \(f\) is real-valued and slice, \(f\)
is slice intrinsic. Hence, for every \(I\in\mathbb S\), the restriction
\(f_I\) is a nonnegative \(\mathbb C_I\)-valued function on \(\mathbb C_I\).

By Theorem~\ref{thm:Berezin-slice}, for every \(\gamma>0\) and every
\(z\in\mathbb C_I\),
\[
(B_\gamma f)(z)=B_{\gamma,I}(f_I)(z),
\]
where \(B_{\gamma,I}\) denotes the classical Berezin transform on the complex
Fock space over \(\mathbb C_I\). Since \(f\) is real-valued and slice intrinsic, each restriction \(f_I\) is a
nonnegative real-valued measurable function on the complex plane \(\mathbb C_I\).
Thus the classical complex result applies on every slice.

Indeed, by the classical parameter-independence theorem for Berezin transforms
on complex Fock spaces, see \cite[Theorem~3.27]{MR2934601}, for every
\(1\le p\le\infty\),
\[
B_{\alpha,I}(f_I)\in L^p(\mathbb C_I,\mathrm dm_{2,I})
\quad\Longleftrightarrow\quad
B_{\beta,I}(f_I)\in L^p(\mathbb C_I,\mathrm dm_{2,I}),
\]
with constants independent of the slice \(I\).  Integrating the corresponding estimates over
\(\mathbb S\) with respect to \(\mathrm d\sigma(I)\), and using the definition
of \(L_s^p(\mathbb H,\mathrm dV)\), gives
\[
B_\alpha f\in L_s^p(\mathbb H,\mathrm dV)
\quad\Longleftrightarrow\quad
B_\beta f\in L_s^p(\mathbb H,\mathrm dV).
\]
This proves \textup{(1)}.

For the vanishing-at-infinity assertion, the same classical result gives, on
each slice \(\mathbb C_I\),
\[
B_{\alpha,I}(f_I)\in C_0(\mathbb C_I)
\quad\Longleftrightarrow\quad
B_{\beta,I}(f_I)\in C_0(\mathbb C_I).
\]
Since \(B_\alpha f\) and \(B_\beta f\) are slice functions by
Theorem~\ref{thm:Berezin-slice}, this is precisely
\[
B_\alpha f\in C_0(\mathbb H)
\quad\Longleftrightarrow\quad
B_\beta f\in C_0(\mathbb H).
\]
This proves \textup{(2)}.
\end{proof}

\begin{remark}\label{rem 6.2}
In contrast with the complex case, one does not in general have $\langle T_f k_z, k_z\rangle_\alpha =\widetilde{f}(z)$. Indeed, for any $z\in\mathbb{C}_I$,
\[
\begin{split} 
 \langle T_f k_z, k_z\rangle_\alpha
&= \langle f \star k_z, k_z\rangle_{\alpha,I}
= \int_{\mathbb{C}_I} \overline{k_z}\, (f \star k_z)\, \mathrm{d}\lambda_{\alpha,I} \\
&= \int_{\mathbb{C}_I} \overline{k_z}(w)\big(F(w)+G(w)J\big)\star k_z(w)\, \mathrm{d}\lambda_{\alpha,I}(w) \\
&= \int_{\mathbb{C}_I} \overline{k_z}(w) F(w) k_z(w)\, \mathrm{d}\lambda_{\alpha,I}(w)
   + \int_{\mathbb{C}_I} \overline{k_z}(w) G(w)\overline{k_z(\overline{w})}\, \mathrm{d}\lambda_{\alpha,I}(w)\cdot J \\
&= \int_{\mathbb{C}_I} |k_z(w)|^2 F(w)\, \mathrm{d}\lambda_{\alpha,I}(w)
   + \int_{\mathbb{C}_I} \overline{k_z}(w) G(w) k_{\overline{z}}(w)\, \mathrm{d}\lambda_{\alpha,I}(w)\cdot J, 
\end{split}
\]which differs from \((B_\alpha f)(z)\) in general.

If \(f\) is intrinsic, equivalently \(G\equiv 0\), then the second term vanishes and we recover
\[
\langle T_f k_z,k_z\rangle_\alpha=(B_\alpha f)(z).
\]
\end{remark}
   
 
Motivated by the notion of \(\mathrm{BMO}\) for quaternionic Hardy spaces introduced in \cite{MR3571445}, we define the corresponding spaces \(\mathrm{BMO}_r^p\) in the setting of quaternionic Fock spaces.
\begin{definition}
 For any radius $r>0$ and every exponent $p\in[1,\infty)$, let $f:\mathbb H\to\mathbb H$ be a function which is locally integrable on each slice $\mathbb C_I$. For each $I\in\mathbb S$ and $z\in\mathbb C_I$, define
\[
\widehat{f}_{r,I}(z):=\frac{1}{\pi r^2}\int_{B_I(z,r)} f_I(w)\,\mathrm{d}m_{2,I}(w).
\]

We say that $f$ belongs to $\mathrm{BMO}^p_{r,I}$ if
\[
\|f\|_{\mathrm{BMO}^p_{r,I}}
:=
\sup_{z\in\mathbb{C}_I}
\left[
\frac{1}{\pi r^2}\int_{B_I(z,r)}
\bigl|f_I(w)-\widehat{f}_{r,I}(z)\bigr|^p
\,\mathrm{d}m_{2,I}(w)
\right]^{1/p}
<+\infty.
\]

We say that $f$ belongs to $\mathrm{BMO}^p_r$ if
\[
\|f\|_{\mathrm{BMO}^p_r}
:=
\sup_{I\in\mathbb S}\|f\|_{\mathrm{BMO}^p_{r,I}}
<+\infty.
\]
For later use, we denote by \(\mathrm{BMO}_r^p(\mathbb C_I)\) the classical bounded
mean-oscillation space on the complex plane \(\mathbb C_I\).
\end{definition}
 
Fix \(I\in\mathbb S\) and choose \(J\in\mathbb S\) with \(J\perp I\). If
\[
f_I=F+GJ
\qquad\text{on }\mathbb C_I,
\]
where \(F,G:\mathbb C_I\to\mathbb C_I\), then
\[
f\in \mathrm{BMO}_{r,I}^p
\quad\Longleftrightarrow\quad
F,G\in \mathrm{BMO}_r^p(\mathbb C_I),
\]
with equivalence of the corresponding seminorms. Since the classical spaces \(\mathrm{BMO}_r^p(\mathbb C_I)\) are independent of
the radius \(r>0\), with equivalent seminorms, the same is true for
\(\mathrm{BMO}_r^p\). We therefore write simply \(\mathrm{BMO}^p\).
  
  Let \(f\) be a slice function on \(\mathbb H\) which is locally integrable on each slice.
  The lemma below shows that, when \(f\) is slice, the family
  \(\{\widehat f_{r,I}\}_{I\in\mathbb S}\) is compatible across slices and therefore
  determines a globally well-defined slice function on \(\mathbb H\). Hence, for
   slice function $f$, we may define
  \[
  \widehat f_r(z) =\widehat f_{r,I}(z)
  \qquad \text{whenever } z\in\mathbb C_I.
  \]

  \begin{lemma}
  	Let \(r>0\). If \(f\) is a slice function on \(\mathbb H\) which is locally integrable on each slice, then the slice-wise averages \(\widehat f_{r,I}\) determine a globally defined slice function on \(\mathbb H\). 
  	Moreover, if $f$ is intrinsic, then $\widehat f_r$ is intrinsic.
  \end{lemma}
  
  \begin{proof}
  	Write $z=x+yI\in\mathbb{C}_I$ and $w=x'+y'I\in\mathbb{C}_I$ with $x,x',y,y'\in \mathbb{R}$.
  	Since $f$ is a slice function, there exist stem functions $\Phi=(a,b)$ 
  	such that $f(x'+y'I)=a(x',y')+I\,b(x',y')$ for every $I\in\mathbb S$.
  	Hence
  	\[
  	\widehat f_r(x+yI)
  	=\frac{1}{\pi r^2}\!\!\int_{(x'-x)^2+(y'-y)^2<r^2}\!\!\!\!\bigl(a(x',y')+I\,b(x',y')\bigr)\,\mathrm d x'\,\mathrm d y'.
  	\]
  	With the change of variables $u=x'-x$, $v=y'-y$, define
  	\[
  	a_r(x,y):=\frac{1}{\pi r^2}\int_{u^2+v^2<r^2}a(x+u,y+v)\,\mathrm d u\,\mathrm d v,
  	\qquad
  	b_r(x,y):=\frac{1}{\pi r^2}\int_{u^2+v^2<r^2}b(x+u,y+v)\,\mathrm d u\,\mathrm d v.
  	\]
  	Then
  	\[
  	\widehat f_r(x+yI)=a_r(x,y)+I\,b_r(x,y)\qquad(I\in\mathbb S).
  	\]
  	By the symmetry of the disk and the parity of \(a,b\), we get
  	\[
  	a_r(x,-y)=a_r(x,y),\qquad b_r(x,-y)=-b_r(x,y),
  	\]
  	hence \(b_r(x,0)=0\) and the values on \(\mathbb R\) do not depend on the choice of \(I\).
  	Therefore, \(\widehat f_r\) is a slice function on \(\mathbb H\).
  	If $f$ is intrinsic, then $a,b$ are real-valued, hence so are $a_r,b_r$, and $\widehat f_r$ is intrinsic.
  \end{proof}
 

\section{Boundedness and Compactness of Toeplitz Operators}\label{sec:toeplitz}
Recall that the orthogonal projection
\[
P_\alpha : L^2_{s,\alpha} \to F^2_\alpha
\]
is an integral operator 
\[
P_\alpha f(z) = \int_{\mathbb{H}} K_\alpha(z,w)\, f(w)\, \mathrm{d}\lambda_\alpha(w).
\]

As recalled in Section~\ref{sec:preliminaries}, for bounded slice-function
symbols the Toeplitz operator is defined by
\[
T_f g=P_\alpha(f\star g),\qquad g\in F_\alpha^2.
\]
For unbounded slice-function symbols satisfying the kernel-integrability
condition \((I_1)\), the same formula is understood initially on the dense
subspace
\[
\mathcal D:=\operatorname{span}_{\mathbb H}
\{K_\alpha(\cdot,a):a\in\mathbb H\}.
\]
Assume that \(\mu\) is a real-valued Borel measure on \(\mathbb H\) such that
\begin{equation}\label{eq:mu1}
\int_{\mathbb H} |K_\alpha(z,w)|\,\mathrm e^{-\alpha|w|^2}\,\mathrm d|\mu|(w)<\infty,
\qquad z\in\mathbb H.
\end{equation}
In the complex Fock space, the reproducing kernel has the  exponential form 
and the analogue of \eqref{eq:mu1} is equivalent to the corresponding square-integrability
condition. 
In the quaternionic setting, however, the reproducing kernel \(K_\alpha(z,w)\) is defined by
a slice power series rather than by such a simple exponential formula. Consequently,
\eqref{eq:mu1} and  
\eqref{eq:intro-mu2} 
are no longer equivalent in general. Nevertheless, \eqref{eq:intro-mu2} always implies
\eqref{eq:mu1} by the Cauchy--Schwarz inequality.

Let $\mu$ be a real-valued Borel measure satisfying \eqref{eq:intro-mu2}. We define the Toeplitz operator with measure symbol $\mu$, initially on a natural dense subspace of $F_\alpha^2$, by
\[
T_\mu g(z)=\int_\HH K_\alpha (z,w) g(w) \mathrm{e}^{-\alpha |w|^2} \mathrm{d} \mu(w), \qquad g \in \mathcal{D}.
\]
Under \eqref{eq:intro-mu2}, then $T_\mu$ is well defined on a dense subspace of $F_\alpha^2$. Indeed, 
If
\[
g(w)=\sum_{k=1}^N K_\alpha(w,z_k)c_k
\]
belongs to $\mathcal D$, then for each fixed $z\in\mathbb H$ the integral defining $T_\mu g(z)$ is absolutely convergent by \eqref{eq:intro-mu2} and the Cauchy--Schwarz inequality. Since $\mathcal D$ is dense in $F_\alpha^2$, this provides a natural dense domain for $T_\mu$.
 
If $\mathrm{d}\mu=\frac{2 \alpha}{\pi} f\mathrm{d}V$, where $f $ is a  real-valued  slice function, then $
T_\mu=T_f$. 
Indeed, since \(f\) is real-valued, \(f\) is slice intrinsic and the slice product $f\star g$ coincides with the pointwise product $fg$.  
Therefore
\[
T_\mu g(z)
=
\int_{\mathbb H}K_\alpha(z,w)g(w) \mathrm e^{-\alpha|w|^2}\,\mathrm d\mu(w)
=
P_\alpha(fg)(z)
=
P_\alpha(f\star g)(z)
=
T_fg(z).
\]

Before turning to the boundedness and compactness criteria, we record some
basic algebraic properties of Toeplitz operators with slice-function symbols.
\begin{proposition}
	Let \(f\in L_s^\infty(\mathbb H,\mathrm{d}V)\). Then \(T_f\) extends to a right-linear bounded operator on \(F_\alpha^2\), and
	\[
	\|T_f\|\le \sqrt2\,\|f\|_{L_s^\infty(\mathrm dV)} .
	\]
\end{proposition}

\begin{proof}
	Fix \(I\in\mathbb S\) and choose \(J\in\mathbb S\) with \(I\perp J\). 
	Write, on the slice \(\mathbb C_I\),
	\[
	f_I=A+BJ,\qquad g_I=F+GJ,
	\]
	where \(A,B:\mathbb C_I\to\mathbb C_I\) are measurable and essentially bounded, while
	\(F,G:\mathbb C_I\to\mathbb C_I\) are holomorphic. By the slice-product formula,
	\[
	(f\star g)_I
	=
	\bigl(AF-B\,\overline{G(\bar z)}\bigr)
	+
	\bigl(AG+B\,\overline{F(\bar z)}\bigr)J .
	\]
By the Cauchy--Schwarz inequality,
	\[
	\begin{aligned}
		|(f\star g)_I(z)|^2
		&\le
		\bigl(|A(z)|^2+|B(z)|^2\bigr)
		\Bigl(
		|F(z)|^2+|G(z)|^2
		+|F(\bar z)|^2+|G(\bar z)|^2
		\Bigr)   \\
		&\le
		\|f\|_\infty^2
		\Bigl(
		|F(z)|^2+|G(z)|^2
		+|F(\bar z)|^2+|G(\bar z)|^2
		\Bigr).
	\end{aligned}
	\]
	Integrating over \(\mathbb C_I\) with respect to \(\mathrm d\lambda_{\alpha,I}\), and using the invariance of the Gaussian measure under \(z\mapsto\bar z\), gives
	\[
	\|f\star g\|_{2,\alpha,I}^2
	\le
	2\|f\|_\infty^2
	\int_{\mathbb C_I}\bigl(|F(z)|^2+|G(z)|^2\bigr)\,\mathrm d\lambda_{\alpha,I}(z)
	=
	2\|f\|_\infty^2\|g\|_{2,\alpha,I}^2 .
	\]
	Since the global \(F_\alpha^2\)-norm agrees with the slice norm, and since
	\(P_\alpha:L^2_{s,\alpha}\to F_\alpha^2\) is an orthogonal projection, we have
	\[
	\|T_f g\|_{2,\alpha}
	=
	\|P_\alpha(f\star g)\|_{2,\alpha}
	\le
	\|f\star g\|_{2,\alpha}
	\le
	\sqrt2\,\|f\|_\infty\,\|g\|_{2,\alpha}.
	\]
	Thus \(T_f\) extends boundedly to \(F_\alpha^2\)  with
	\[
	\|T_f\|\le \sqrt2\,\|f\|_{L_s^\infty(\mathrm dV)} .
	\]
\end{proof}

We next record some elementary algebraic properties of Toeplitz operators. In contrast with the complex case, these properties are affected by the noncommutativity of the quaternionic product.

If \(f\) is a real-valued nonnegative slice function, then \(T_f\) is positive. Moreover, for bounded slice functions \(f\) and \(g\), one clearly has
\[
T_{f+g}=T_f+T_g.
\] 
The behavior under quaternionic scalar multiplication is more delicate. Let
\(a\in\mathbb H\), and regard \(a\) as a constant slice function. If the right
multiplication of the symbol is interpreted as the slice product \(f\star a\),
then associativity gives
\[
T_{f\star a}=T_f\circ L_a,
\qquad
L_ag:=a\star g.
\]
In general, \(a\star g\) is not the pointwise product \(ga\). Thus this relation
is not an ordinary scalar multiplication of the operator \(T_f\) unless \(a\in\mathbb R\). Accordingly, several arguments from the classical complex
theory do not carry over verbatim to the quaternionic setting.
 
In contrast with the complex case, the identity $
T_f^*=T_{\overline f}$ 
fails in general for bounded slice function symbols \(f\). More generally, if the adjoint \((T_f)^*\) is again a Toeplitz operator \(T_g\) with bounded slice function symbol \(g\), then \(g\) need not coincide with \(\overline f\). The next theorem characterizes exactly which bounded slice function symbol \(g\)  satisfies $
(T_f)^*=T_g$.
\begin{theorem}\label{thm:adjoint-characterization}
Suppose \(f\) and \(g\) are bounded slice functions on \(\mathbb H\).
For \(I\in\mathbb S\) and \(J\in\mathbb S\) with \(J\perp I\), write on \(\mathbb C_I\)
\[
f_I=A+B\,J,\qquad g_I=C+D\,J,
\]
where \(A,B,C,D:\mathbb C_I\to\mathbb C_I\) are measurable and essentially bounded. Then the following are equivalent:
\begin{enumerate}
\item[(a)] \((T_f)^*=T_g\) on \(F_\alpha^2\).

\item[(b)] For one, equivalently for every, choice of \(I\in\mathbb S\) and
\(J\in\mathbb S\) with \(J\perp I\), the corresponding decompositions above satisfy
\[
A(z)=\overline{C(z)}
\quad\text{and}\quad
B(z)=-D(\bar z)
\qquad\text{for a.e. } z\in\mathbb C_I .
\]
\end{enumerate}
\end{theorem}

\begin{proof}
It is enough to prove that condition \textup{(b)} for one pair \(I,J\) implies \textup{(a)}, and that \textup{(a)} implies condition \textup{(b)} for every pair \(I,J\).

Assume first that condition \textup{(b)} holds for some \(I\in\mathbb S\) and some \(J\perp I\). We prove \((T_f)^*=T_g\). 
Fix $\psi,\phi\in F^2_{\alpha}$ and write $\psi_I=F+GJ$, $\phi_I=H+KJ$ on $\mathbb{C}_I$ with $F,G,H,K:\mathbb{C}_I\to\mathbb{C}_I$.
Using the slice $\star$–formula,
\[
f_I\star \psi_I
=
\bigl(AF-B\,\overline{G(\bar z)}\bigr)
+\bigl(AG+B\,\overline{F(\bar z)}\bigr)J
=:X+YJ,
\]
\[
g _I\star \phi_I
=
\bigl(CH-D\,\overline{K(\bar z)}\bigr)
+\bigl(CK+D\,\overline{H(\bar z)}\bigr)J
=:U+VJ.
\] 
     
     With $J a=\overline a\,J$ and $J a J=-\,\overline a$ for $a\in\mathbb{C}_I$, one computes
\[
\overline \phi_I \cdot \,(f_I\star \psi_I)
=
\bigl(\overline H\,X+K\,\overline Y\bigr)
+\bigl(\overline H\,Y-K\,\overline X\bigr)J,
\quad
(\overline{g _I\star \phi _I})\,\cdot \psi_I
=
\bigl(\overline U\,F+V\,\overline G\bigr)
+\bigl(\overline U\,G-V\,\overline F\bigr)J,
\]
where $\overline U=\overline H\,\overline C - K(\bar z)\,\overline D$ and $\overline V=\overline K\,\overline C - \overline{H(\bar z)}\,\overline D$.
Therefore
\[
\langle f\star \psi,\,\phi\rangle_{\alpha,I}
=
\int_{\mathbb{C}_I}\!\bigl(\overline H\,X+K\,\overline Y\bigr)\,\mathrm d\lambda_{\alpha,I}
+\int_{\mathbb{C}_I}\!\bigl(\overline H\,Y-K\,\overline X\bigr)\,\mathrm d\lambda_{\alpha,I}\cdot J,
\]
\[
\langle \psi,\,g \star \phi \rangle_{\alpha,I}
=
\int_{\mathbb{C}_I}\!\bigl(\overline U\,F+V\,\overline G\bigr)\,\mathrm d\lambda_{\alpha,I}
+\int_{\mathbb{C}_I}\!\bigl(\overline U\,G-V\,\overline F\bigr)\,\mathrm d\lambda_{\alpha,I}\cdot J.
\]

After substituting $A=\overline C$ and $B(z)=-D(\bar z)$, the terms involving $A$ and $C$ cancel immediately. For the remaining terms, one performs the change of variables $t=\bar z$ and uses the identities
\[
B(t)=-D(\bar t), 
\]
together with the invariance of $\lambda_{\alpha,I}$ under conjugation, to obtain cancellation term by term. The same argument applies to the $\mathbb C_IJ$-component. 
Therefore
\[
\langle f_I\star\psi_I,\phi_I\rangle_{\alpha,I}
=
\langle \psi_I,g_I\star\phi_I\rangle_{\alpha,I}.
\]
Using the compatibility between the global and slice inner products, we obtain
\[
\langle f\star\psi,\phi\rangle_\alpha
=
\langle \psi,g\star\phi\rangle_\alpha.
\]
Since $P_\alpha$ is self-adjoint and $\psi,\phi\in F_\alpha^2$, it follows that
\[
\langle T_f\psi,\phi\rangle_\alpha
=
\langle \psi,T_g\phi\rangle_\alpha.
\]
Hence $(T_f)^*=T_g$.

Conversely, assume that \((T_f)^*=T_g\). Restricting the identity
\[
\langle T_f\psi,\phi\rangle_\alpha=\langle \psi,T_g\phi\rangle_\alpha
\]
to a fixed slice $\mathbb C_I$ and using the compatibility between the global and slice inner products, we obtain
\[
\langle f_I\star \psi_I,\phi_I\rangle_{\alpha,I}
=
\langle \psi_I,g_I\star \phi_I\rangle_{\alpha,I}.
\]
With the above decompositions and $\star$–expansions,  
taking $G\equiv 0$ and $K\equiv 0$ gives
\[
\int_{\mathbb{C}_I}\overline H\,(A-\overline C)\,F\,\mathrm d\lambda_{\alpha,I}=0
\qquad\text{for all }F,H\in  \mathcal F_\alpha^2(\mathbb C_I).
\]
Set
\[
h:=A-\overline C.
\]
Then the above identity shows that
\[
\int_{\mathbb C_I}\overline{H(w)}\,h(w)\,F(w)\,\mathrm d\lambda_{\alpha,I}(w)=0
\qquad\text{for all }F,H\in \mathcal F_\alpha^2(\mathbb C_I).
\]
Equivalently,
\[
\langle T_{h,I}F,H\rangle_{\alpha,I}=0
\qquad\text{for all }F,H\in \mathcal F_\alpha^2(\mathbb C_I),
\]
where \(T_{h,I}\) denotes the Toeplitz operator on the complex Fock space \(\mathcal F_\alpha^2(\mathbb C_I)\) with symbol \(h\).
Hence
\[
T_{h,I}=0.
\]
By the injectivity of Toeplitz quantization on the analytic Fock space; see \cite[Proposition~3.8]{MR4192830}, it follows that
\[
h=0 \quad\text{a.e. on }\mathbb C_I.
\]
Therefore
\[
A=\overline C
\qquad\text{a.e. on }\mathbb C_I.
\]
With $A=\overline C$, take $F\equiv K\equiv 0$. Then the equality of the $\mathbb C_I$-components reduces to
\[
-\int_{\mathbb C_I}\overline{H(w)}\,B(w)\,\overline{G(\bar w)}\,\mathrm d\lambda_{\alpha,I}(w)
=
\int_{\mathbb C_I}D(w)\,\overline{H(\bar w)}\,\overline{G(w)}\,\mathrm d\lambda_{\alpha,I}(w).
\]
After the change of variables $u=\bar w$ in the right-hand side and using the invariance of $\lambda_{\alpha,I}$ under conjugation, we obtain
\[
\int_{\mathbb C_I}\bigl(B(w)+D(\bar w)\bigr)\,\overline{H(w)}\,\overline{G(\bar w)}\,\mathrm d\lambda_{\alpha,I}(w)=0
\]
for all $G,H\in \mathcal F_{\alpha}^2(\mathbb C_I)$.
  
By the same injectivity argument, we obtain
\[
D(u)=-B(\bar u)
\qquad\text{for a.e. }u\in\mathbb C_I.
\]
Equivalently,
\[
B(w)=-D(\bar w)
\qquad\text{for a.e. }w\in\mathbb C_I.
\]

\end{proof} 

\begin{corollary}
Let $f$ be a bounded slice function. Fix $I\in\mathbb S$ and choose $J\in\mathbb S$ with $J\perp I$. Write
\[
f_I=A+BJ
\]
on $\mathbb C_I$, where $A,B:\mathbb C_I\to\mathbb C_I$ are measurable and essentially bounded. If
\[
B(z)=B(\bar z)
\qquad\text{for a.e. }z\in\mathbb C_I,
\]
then
\[
T_f^*=T_{\overline f}
\quad\text{on }F_\alpha^2.
\]
In particular, if $f$ is slice intrinsic, then $
T_f^*=T_{\overline f}$.
\end{corollary}

\begin{proof}
Write
\[
(\overline f)_I=\overline{A}-BJ.
\]
Thus, in the notation of Theorem~\ref{thm:adjoint-characterization}, we have
\[
C=\overline A,\qquad D=-B.
\]
Hence
\[
A=\overline C
\quad\text{and}\quad
B(z)=-D(\bar z)=B(\bar z)
\]
for a.e. \(z\in\mathbb C_I\). The conclusion therefore follows from Theorem~\ref{thm:adjoint-characterization}. The intrinsic case corresponds to \(B\equiv 0\).
\end{proof}
Throughout the remainder of this subsection, whenever $\mu$ is a positive Borel measure on $\mathbb H$, we assume that
\[
\mu^1:=\mu|_{\mathbb H\setminus\mathbb R}
\]
is a Radon measure on $\mathbb H\setminus\mathbb R$.
This is precisely the regularity hypothesis required in the Fock--Carleson measure characterizations established in Section~\ref{s:4}.
 
\begin{proof}[Proof of Theorem~\ref{thm:intro-mu-bounded}]
For \(g\in F_\alpha^2\), the definition of \(T_\mu\) gives
\[
\langle T_\mu g,g\rangle_\alpha
=
\int_{\mathbb H}|g(w)|^2 \mathrm e^{-\alpha|w|^2}\,\mathrm d\mu(w)
=
\int_{\mathbb H}\bigl|g(w) \mathrm e^{-\frac{\alpha}{2}|w|^2}\bigr|^2\,\mathrm d\mu(w).
\]
Since \(\mu\) is positive, \(T_\mu\) is bounded on \(F_\alpha^2\) if and only if there exists \(C>0\) such that
\[
\int_{\mathbb H}\bigl|g(w) \mathrm e^{-\frac{\alpha}{2}|w|^2}\bigr|^2\,\mathrm d\mu(w)
\le C\|g\|_{2,\alpha}^2,
\qquad g\in F_\alpha^2.
\]
That is, \(T_\mu\) is bounded if and only if \(\mu\) is a quaternionic Fock--Carleson measure for \(F_\alpha^2\). The equivalence of \textup{(a)}, \textup{(b)}, and \textup{(c)} now follows from Theorem~\ref{thm:Fock--Carleson}.
\end{proof}

\begin{proof}[Proof of Theorem~\ref{thm:intro-mu-compact}]
Since \(\mu\) is positive, \(T_\mu\) is compact on \(F_\alpha^2\) if and only if \(\mu\) is a vanishing quaternionic Fock--Carleson measure for \(F_\alpha^2\). Therefore the conclusion follows directly from Theorem~\ref{thm:vanishing-Carleson}.
\end{proof}

Recall that for $z=x +y  J\in\mathbb H,y \geq 0$ we write
\[
S(z,r):=\bigcup_{I\in\mathbb S} B_I(x+y I,r),
\]
and that $\mathrm dV=\mathrm d m_{2,I}\,\mathrm d\sigma(I)$ with $\sigma(\mathbb S)=1$.
\begin{lemma}\label{lemma:real-intrinsic-average}
Let $r>0$, and let $f:\HH\to\RR$ be a real-valued slice function.
For $z=x +y J\in\HH$, with $x \in\RR$, $y \ge 0$, and $J\in\mathbb S$, set
\[
z_I:=x +y I\in\mathbb C_I,\qquad I\in\mathbb S.
\]
Then, for every $I\in\mathbb S$,
\[
\int_{S(z,r)} f(w)\,\mathrm dV(w)
=
\int_{B_I(z_I,r)} f(w)\,\mathrm d m_{2,I}(w).
\]
\end{lemma}

\begin{proof}
Since $f$ is real-valued and slice, we have $f$ is slice intrinsic. Hence there exist real-valued functions $a,b:\RR^2\to\RR$ such that
\[
f(x+yI)=a(x,y)+I\,b(x,y),
\qquad x,y\in\RR,\ I\in\mathbb S.
\]
Because $f(x+yI)\in\RR$ for all $x,y\in\RR$ and all $I\in\mathbb S$, we must have $b\equiv 0$. Therefore
\[
f(x+yI)=a(x,y),\qquad x,y\in\RR,\ I\in\mathbb S,
\]
so the value of $f(x+yI)$ depends only on the pair $(x,y)$ and is independent of $I$.

Assume first that $z\notin\RR$, and write
\[
z=x_0+y_0J,\qquad x_0\in\RR,\quad y_0>0,\quad J\in\mathbb S.
\]
Let
\[
D_r:=\{(x,y)\in\RR^2:(x-x_0)^2+(y-y_0)^2\le r^2\}.
\]
For each $I\in\mathbb S$, the map $(x,y)\mapsto x+yI$ identifies $D_r$ with the disc $B_I(z_I,r)\subset\mathbb C_I$. Since $f(x+yI)=a(x,y)$, we obtain
\[
\int_{B_I(z_I,r)} f(w)\,\mathrm d m_{2,I}(w)
=
\int_{D_r} a(x,y)\,\mathrm dx\,\mathrm dy,
\]
and the right-hand side is independent of $I$.

Since the real axis has $\mathrm dV$-measure zero, integration over $S(z,r)$ may be identified with integration over $S(z,r)\setminus\mathbb R$. By the definition of $\mathrm dV$ on $\HH\setminus\RR$,  we therefore get
\[
\int_{S(z,r)} f(w)\,\mathrm dV(w)
=
\int_{\mathbb S}\int_{B_I(z_I,r)} f(w)\,\mathrm d m_{2,I}(w)\,\mathrm d\sigma(I)
=
\int_{\mathbb S}\left(\int_{D_r} a(x,y)\,\mathrm dx\,\mathrm dy\right)\mathrm d\sigma(I).
\]
Since $\sigma(\mathbb S)=1$, it follows that
\[
\int_{S(z,r)} f(w)\,\mathrm dV(w)
=
\int_{D_r} a(x,y)\,\mathrm dx\,\mathrm dy
=
\int_{B_I(z_I,r)} f(w)\,\mathrm d m_{2,I}(w)
\]
for every $I\in\mathbb S$.

If $z\in\RR$, the same argument applies with
\[
D_r:=\{(x,y)\in\RR^2:(x-z)^2+y^2\le r^2\},
\]
and $z_I=z\in\mathbb C_I$ for every $I\in\mathbb S$. This proves the result.
\end{proof}

\begin{theorem}\label{thm:f1}
Suppose $f$ is a real-valued nonnegative slice function satisfying $(I_1)$. Then the following are equivalent:
\begin{enumerate}[(a)]
\item $T_f$ is bounded on $F^2_\alpha$;
\item $\widetilde{f}=B_\alpha f \in L_s^\infty(\mathbb{H},\mathrm{d}V)$;
\item \(B_\beta f \in L_s^\infty(\mathbb H,\mathrm dV)\) for some, equivalently
for every, \(\beta>0\);
\item \(\widehat f_r \in L_s^\infty(\mathbb H,\mathrm dV)\) for some,
equivalently for every, \(r>0\).\end{enumerate}
\end{theorem}

\begin{proof}
Set
\[
\mathrm d\mu(w):=\frac{2\alpha}{\pi}f(w)\,\mathrm dV(w).
\]
Since \(f\) is real-valued and nonnegative, \(\mu\) is a positive Borel measure
on \(\mathbb H\). Its restriction \(\mu^1\) is Radon, because \(f\) is locally
integrable with respect to \(\mathrm dV\) on \(\mathbb H\setminus\mathbb R\) by
\((I_1)\), and \(\mathrm dV\) is Radon there. Moreover,
the assumption \((I_1)\) implies that \(\mu\) satisfies \eqref{eq:intro-mu2}. Hence
Theorem~\ref{thm:intro-mu-bounded} applies to \(\mu\).
Since \(f\) is real-valued, one has \(f\star g=fg\), and therefore $
T_\mu=T_f $. 
Thus \(T_f\) is bounded if and only if the equivalent conditions in
Theorem~\ref{thm:intro-mu-bounded} hold for the measure \(\mu\).

It remains only to identify those conditions in terms of \(f\). Since \(f\) is
real-valued and slice, we have \(f\) is slice
intrinsic. By Lemma~\ref{lemma:real-intrinsic-average},
\[
\widehat\mu_r(z)
=
\frac{\mu(S(z,r))}{\pi r^2}
=
\frac{2\alpha}{\pi}\widehat f_r(z).
\]
Hence the boundedness of the symmetric box averages in
Theorem~\ref{thm:intro-mu-bounded} is equivalent to
\[
\widehat f_r\in L_s^\infty(\mathbb H,\mathrm dV).
\] 

On the other hand, since \(T_\mu=T_f\) and \(f\) is slice intrinsic,
Remark~\ref{rem 6.2} gives
\[
\int_{\mathbb H}
\bigl|k_z(w)\mathrm e^{-\alpha|w|^2/2}\bigr|^2
\,\mathrm d\mu(w)
=
\langle T_\mu k_z,k_z\rangle_\alpha
=
\langle T_f k_z,k_z\rangle_\alpha
=
(B_\alpha f)(z).
\]
Therefore condition \textup{(b)} of Theorem~\ref{thm:intro-mu-bounded} is equivalent to $
B_\alpha f \in L_s^\infty(\mathbb H,\mathrm dV)$.  
Together with the preceding identification of \(\widehat\mu_r\) and
\(\widehat f_r\), this proves the equivalence of
\textup{(a)}, \textup{(b)}, and \textup{(d)}.

Finally, the equivalence of \textup{(b)} and \textup{(c)} follows from
Theorem~\ref{thm 5.9}. Therefore all four conditions are equivalent.
\end{proof}

\begin{theorem}\label{thm 7.5}
Suppose $f $    is a real-valued nonnegative slice function satisfying $(I_1)$. Then the following are equivalent:
\begin{enumerate}[(a)]
\item $T_f$ is compact on $F^2_\alpha$;
\item $\widetilde{f} \in C_0(\mathbb{H})$;
\item $B_\beta f \in C_0(\mathbb{H})$ for some, equivalently 
for every, $\beta>0$;
\item $\widehat{f}_r \in C_0(\mathbb{H})$ for some, equivalently 
for every, $r>0$.
\end{enumerate}
\end{theorem}

\begin{proof}
Let $
\mathrm{d}\mu(w):=\frac{2\alpha}{\pi}f(w)\,\mathrm{d}V(w)$. As in the proof of Theorem~\ref{thm:f1}, the measure $\mu$ satisfies the hypotheses of Theorem~\ref{thm:intro-mu-compact}, and $T_\mu=T_f$.

Moreover, the averaging function \(\widehat f_r\) is slice continuous by the
local integrability of \(f\) on each slice. Hence, using
Lemma~\ref{lemma:real-intrinsic-average}, the condition
\(\widehat f_r\in C_0(\mathbb H)\) is equivalent to the vanishing averaging
condition for \(\widehat\mu_r\) in Theorem~\ref{thm:intro-mu-compact}. Therefore Theorem~\ref{thm:intro-mu-compact} gives the equivalence of 
\textup{(a)}, \textup{(b)}, and \textup{(d)}.

Finally, the equivalence of \textup{(b)} and \textup{(c)} follows from Theorem~\ref{thm 5.9}.
\end{proof}

\begin{remark}
For real-valued nonnegative slice function symbols, the boundedness and compactness of $T_f$ reduce essentially to the corresponding questions on the complex Fock space over each slice $\mathbb C_I$. 
\end{remark}

\begin{proof}[Proof of Theorem~\ref{thm:intro-bmo-bounded}]
	We first prove \textup{(a)}\(\Rightarrow\)\textup{(b)}. 
	Assume that \(T_f\) is bounded on \(F_\alpha^2\). Fix \(I\in\mathbb S\) and choose
	\(J\in\mathbb S\) with \(J\perp I\). Write
	\[
	f_I=F+GJ
	\qquad\text{on } \mathbb C_I,
	\]
	where \(F,G:\mathbb C_I\to\mathbb C_I\). For \(z\in\mathbb C_I\), the computation in
	Remark~\ref{rem 6.2} gives
	\[
	\langle T_f k_z,k_z\rangle_\alpha
	=
	B_{\alpha,I}F(z)+H_G(z)J,
	\]
	and
	\[
	\langle T_f k_{\bar z},k_z\rangle_\alpha
	=
	H_F(z)+B_{\alpha,I}G(z)J,
	\]
	where
	\[
	H_G(z)
	:=
	\int_{\mathbb C_I}
	\overline{k_z(w)}\,G(w)\,k_{\bar z}(w)
	\,\mathrm d\lambda_{\alpha,I}(w),
	\]
	and \(H_F\) is defined analogously with \(G\) replaced by \(F\).

	Since \(\|k_z\|_{2,\alpha}=\|k_{\bar z}\|_{2,\alpha}=1\), boundedness of \(T_f\) implies
	\[
	\bigl|\langle T_f k_z,k_z\rangle_\alpha\bigr|
	\le \|T_f\|,
	\qquad
	\bigl|\langle T_f k_{\bar z},k_z\rangle_\alpha\bigr|
	\le \|T_f\|.
	\]
	Moreover, \(B_{\alpha,I}F(z),H_G(z),H_F(z),B_{\alpha,I}G(z)\in\mathbb C_I\). Hence, for
	\(u,v\in\mathbb C_I\),
	\[
	|u+vJ|^2=|u|^2+|v|^2.
	\]
	Applying this identity to the two preceding decompositions, we obtain
	\[
	|B_{\alpha,I}F(z)|\le \|T_f\|,
	\qquad
	|B_{\alpha,I}G(z)|\le \|T_f\|,
	\qquad z\in\mathbb C_I .
	\]
	By the slice-wise representation of the Berezin transform,
	\[
	(B_\alpha f)_I
	=
	B_{\alpha,I}F+B_{\alpha,I}G\,J,
	\]
	and therefore
	\[
	|(B_\alpha f)(z)|
	\le \sqrt{2}\,\|T_f\|,
	\qquad z\in\mathbb C_I.
	\]
	Since \(I\in\mathbb S\) was arbitrary, it follows that
	\[
	B_\alpha f\in L_s^\infty(\mathbb H,\mathrm dV).
	\]
	Thus \textup{(a)} implies \textup{(b)}.
	
	We next prove \textup{(b)}\(\Rightarrow\)\textup{(a)}. 
	Assume that
	\[
	B_\alpha f\in L_s^\infty(\mathbb H,\mathrm dV).
	\]
	Fix \(I\in\mathbb S\) and choose \(J\perp I\). Write
	\[
	f_I=F+GJ,
	\qquad
	F,G:\mathbb C_I\to\mathbb C_I .
	\]
	By Theorem~\ref{thm:Berezin-slice},
	\[
	(B_\alpha f)_I
	=
	B_{\alpha,I}F+B_{\alpha,I}G\,J.
	\]
	Consequently \(B_{\alpha,I}F\) and \(B_{\alpha,I}G\) are bounded on \(\mathbb C_I\). Since
	\(f\in\mathrm{BMO}^1\), the slice decomposition of the \(\mathrm{BMO}^1\)-seminorm gives $
	F,G\in \mathrm{BMO}^1(\mathbb C_I)$.  
	By the classical Fock-space theorem for Toeplitz operators with 
\(\mathrm{BMO}^1\)-symbols, see \cite[Theorem 6.20]{MR2934601}, the boundedness of \(B_{\alpha,I}F\) and
	\(B_{\alpha,I}G\) implies that the complex Toeplitz operators
	\[
	T_{F,I}h=P_{\alpha,I}(Fh),
	\qquad
	T_{G,I}h=P_{\alpha,I}(Gh),
	\qquad h\in  \mathcal F_\alpha^2(\mathbb C_I),
	\]
	are bounded on \( \mathcal F_\alpha^2(\mathbb C_I)\).
	
	Let \(g\in F_\alpha^2\), and write
	\[
	g_I=H+KJ,
	\qquad
	H,K\in \mathcal F_\alpha^2(\mathbb C_I).
	\]
	For \(h\in \mathcal F_\alpha^2(\mathbb C_I)\), set
	\[
	h^\#(z):=\overline{h(\bar z)},\qquad z\in\mathbb C_I.
	\]
	The map \(h\mapsto h^\#\) is an isometric involution on \(\mathcal F_\alpha^2(\mathbb C_I)\). By
	Proposition~\ref{prop:slice_star_formula},
	\[
	(f\star g)_I
	=
	(FH-GK^\#)+(FK+GH^\#)J.
	\]
	Applying \(P_{\alpha,I}\) componentwise yields
	\[
	(T_fg)_I
	=
	\bigl(T_{F,I}H-T_{G,I}K^\#\bigr)
	+
	\bigl(T_{F,I}K+T_{G,I}H^\#\bigr)J.
	\]
	Hence
	\[
	\begin{aligned}
		\|(T_fg)_I\|_{2,\alpha,I}
		&\lesssim
		\|T_{F,I}H\|_{\mathcal F_\alpha^2(\mathbb C_I)}
		+\|T_{G,I}K^\#\|_{\mathcal F_\alpha^2(\mathbb C_I)}  
		+\|T_{F,I}K\|_{\mathcal F_\alpha^2(\mathbb C_I)}
		+\|T_{G,I}H^\#\|_{\mathcal F_\alpha^2(\mathbb C_I)}  \\
		&\lesssim
		\|H\|_{\mathcal F_\alpha^2(\mathbb C_I)}
		+\|K\|_{\mathcal F_\alpha^2(\mathbb C_I)}\approx 
		\|g\|_{2,\alpha}.
	\end{aligned}
	\]
  
	Therefore \(T_f\) is bounded on \(F_\alpha^2\). This proves
	\textup{(b)}\(\Rightarrow\)\textup{(a)}.
	
	It remains to prove the equivalence of \textup{(b)}, \textup{(c)}, and \textup{(d)}.
	Fix \(I\in\mathbb S\) and write \(f_I=F+GJ\) with \(J\perp I\). Since
	\(f\in\mathrm{BMO}^1\), we have $
	F,G\in \mathrm{BMO}^1(\mathbb C_I)$,  
	with seminorms bounded uniformly in \(I\) by the definition of
\(\mathrm{BMO}^1\). Moreover,
	\[
	(B_\gamma f)_I
	=
	B_{\gamma,I}F+B_{\gamma,I}G\,J,
	\qquad \gamma>0,
	\]
	and the local average \(\widehat f_r\) decomposes as
	\[
	(\widehat f_r)_I
	=
	\widehat F_{r }+\widehat G_{r }\,J.
	\]
	The classical complex Fock-space \(\mathrm{BMO}^1\) theory gives, for each fixed
	\(\gamma>0\) and \(r>0\),
	\[
	B_{\alpha,I}F\in L^\infty(\mathbb C_I)
	\Longleftrightarrow
	B_{\gamma,I}F\in L^\infty(\mathbb C_I)
	\Longleftrightarrow
	\widehat F_{r }\in L^\infty(\mathbb C_I),
	\]
	and the same equivalences hold for \(G\).  Taking the supremum over \(I\in\mathbb S\) gives
	\[
	B_\alpha f\in L_s^\infty(\mathbb H,\mathrm dV)
	\Longleftrightarrow
	B_\gamma f\in L_s^\infty(\mathbb H,\mathrm dV)
	\Longleftrightarrow
	\widehat f_r\in L_s^\infty(\mathbb H,\mathrm dV).
	\]
	This proves the equivalence of \textup{(b)}, \textup{(c)}, and \textup{(d)}, and hence
	completes the proof.
\end{proof}

\begin{proof}[Proof of Theorem~\ref{thm:intro-bmo-compact}]
We first assume that \(T_f\) is compact on \(F_\alpha^2\). Fix \(I\in\mathbb S\)
and choose \(J\in\mathbb S\) with \(J\perp I\). Write
\[
f_I=F+GJ
\qquad\text{on }\mathbb C_I,
\]
where \(F,G:\mathbb C_I\to\mathbb C_I\). As in the proof of
Theorem~\ref{thm:intro-bmo-bounded}, for \(z\in\mathbb C_I\) one has
\[
\langle T_f k_z,k_z\rangle_\alpha
=
B_{\alpha,I}F(z)+H_G(z)J,
\]
and
\[
\langle T_f k_{\bar z},k_z\rangle_\alpha
=
H_F(z)+B_{\alpha,I}G(z)J,
\]
where \(H_F\) and \(H_G\) are  defined in the proof of Theorem \ref{thm:intro-bmo-bounded}.

Since \(k_z\) and \(k_{\bar z}\) are bounded in \(F_\alpha^2\) and converge
weakly to \(0\) as \(|z|\to\infty\) on the slice \(\mathbb C_I\), compactness of
\(T_f\) gives
\[
\|T_fk_z\|_{2,\alpha}\to0,
\qquad
\|T_fk_{\bar z}\|_{2,\alpha}\to0 .
\]
Hence
\[
\langle T_f k_z,k_z\rangle_\alpha\to0,
\qquad
\langle T_f k_{\bar z},k_z\rangle_\alpha\to0 .
\]
Using again the orthogonal decomposition
\(\mathbb H=\mathbb C_I\oplus \mathbb C_IJ\), we obtain
\[
B_{\alpha,I}F(z)\to0,
\qquad
B_{\alpha,I}G(z)\to0
\qquad (|z|\to\infty,\ z\in\mathbb C_I).
\]
By Theorem~\ref{thm:Berezin-slice},
\[
(B_\alpha f)_I=B_{\alpha,I}F+B_{\alpha,I}GJ.
\]
Thus \((B_\alpha f)_I\) vanishes at infinity on \(\mathbb C_I\). Since \(I\in
\mathbb S\) was arbitrary, we have \( B_\alpha f\in C_0(\mathbb H)\).

Conversely, assume that \( B_\alpha f\in C_0(\mathbb H)\). Since 
\(C_0(\mathbb H)\subset L_s^\infty(\mathbb H,\mathrm dV)\), Theorem~\ref{thm:intro-bmo-bounded} 
first implies that \(T_f\) is bounded on \(F_\alpha^2\). Fix \(I\in\mathbb S\) 
and choose \(J\perp I\). Write \(f_I=F+GJ\). 
By Theorem~\ref{thm:Berezin-slice},
\[
(B_\alpha f)_I=B_{\alpha,I}F+B_{\alpha,I}GJ.
\]
Hence
\[
B_{\alpha,I}F\in C_0(\mathbb C_I),
\qquad
B_{\alpha,I}G\in C_0(\mathbb C_I).
\]
Since \(f\in\mathrm{BMO}^1\), we also have $
F,G\in\mathrm{BMO}^1(\mathbb C_I)$. 
By the classical compactness criterion for Toeplitz operators with
\(\mathrm{BMO}^1\)-symbols on the complex Fock space
\cite[Theorem 6.27]{MR2934601}, the operators
\[
T_{F,I}h=P_{\alpha,I}(Fh),
\qquad
T_{G,I}h=P_{\alpha,I}(Gh)
\]
are compact on \(\mathcal F_\alpha^2(\mathbb C_I)\).

By Remark \ref{equivalent}, it is enough to test compactness on bounded sequences converging to zero
uniformly on compact subsets. 
Let
\(\{g_n\}\) be a bounded sequence in \(F_\alpha^2\) such that
\(g_n\to0\) uniformly on compact subsets of \(\mathbb H\). By Remark \ref{equivalent}, this is equivalent to compact convergence on the fixed slice
\(\mathbb C_I\). Write
\[
(g_n)_I=H_n+K_nJ,
\qquad
H_n,K_n\in \mathcal F_\alpha^2(\mathbb C_I).
\]
Then \(\{H_n\}\) and \(\{K_n\}\) are bounded in \(\mathcal F_\alpha^2(\mathbb C_I)\) and
converge to \(0\) uniformly on compact subsets of \(\mathbb C_I\). The same is
true for
\[
H_n^\#(z):=\overline{H_n(\bar z)},
\qquad
K_n^\#(z):=\overline{K_n(\bar z)}.
\]
In the complex Fock space, bounded sequences converging uniformly to \(0\) on
compact subsets converge weakly to \(0\). Hence compactness of \(T_{F,I}\) and
\(T_{G,I}\) yields
\[
T_{F,I}H_n\to0,\quad
T_{F,I}K_n\to0,\quad
T_{G,I}H_n^\#\to0,\quad
T_{G,I}K_n^\#\to0
\]
in \(\mathcal F_\alpha^2(\mathbb C_I)\).

Using the slice-product formula as in the proof of Theorem~\ref{thm:intro-bmo-bounded}, we have
\[
(T_fg_n)_I
=
\bigl(T_{F,I}H_n-T_{G,I}K_n^\#\bigr)
+
\bigl(T_{F,I}K_n+T_{G,I}H_n^\#\bigr)J.
\]
Therefore
\[
\|(T_fg_n)_I\|_{2,\alpha,I}\to0.
\]
Since the global and slice \(F_\alpha^2\)-norms agree, we obtain
\[
\|T_fg_n\|_{2,\alpha}\to0.
\]
Thus \(T_f\) is compact on \(F_\alpha^2\). This completes the proof.
\end{proof}
 
\bigskip

\subsection*{Conflict of interest}
The authors have no conflict of interest to declare that are relevant to the content of this article. 
\subsection*{Data availability statement}
No data, models, or code were generated or used for the research described in the article.
 
\bibliographystyle{amsplain}
\bibliography{references}

@incollection {MR3587897,
    AUTHOR = {Alpay, Daniel and Colombo, Fabrizio and Sabadini, Irene and
              Salomon, Guy},
     TITLE = {The {F}ock space in the slice hyperholomorphic setting},
 BOOKTITLE = {Hypercomplex analysis: new perspectives and applications},
    SERIES = {Trends Math.},
     PAGES = {43--59},
 PUBLISHER = {Birkh\"auser/Springer, Cham},
      YEAR = {2014},
      ISBN = {978-3-319-08770-2; 978-3-319-08771-9},
   MRCLASS = {30G35 (30H20)},
  MRNUMBER = {3587897},
}

@book {colombo2016entire,
    AUTHOR = {Colombo, Fabrizio and Sabadini, Irene and Struppa, Daniele C.},
     TITLE = {Entire slice regular functions},
    SERIES = {SpringerBriefs in Mathematics},
 PUBLISHER = {Springer, Cham},
      YEAR = {2016},
     PAGES = {v+118},
      ISBN = {978-3-319-49264-3; 978-3-319-49265-0},
   MRCLASS = {30G35 (30B10 30D15)},
  MRNUMBER = {3585395},
MRREVIEWER = {Alessandro\ Perotti},
       DOI = {10.1007/978-3-319-49265-0},
       URL = {https://doi.org/10.1007/978-3-319-49265-0},
}

@article {MR374886,
    AUTHOR = {Hastings, William W.},
     TITLE = {A {C}arleson measure theorem for {B}ergman spaces},
   JOURNAL = {Proc. Amer. Math. Soc.},
  FJOURNAL = {Proceedings of the American Mathematical Society},
    VOLUME = {52},
      YEAR = {1975},
     PAGES = {237--241},
      ISSN = {0002-9939,1088-6826},
   MRCLASS = {46E15 (30A78 32A30)},
  MRNUMBER = {374886},
MRREVIEWER = {Arne\ Stray},
       DOI = {10.2307/2040137},
       URL = {https://doi.org/10.2307/2040137},
}

@article {MR173012,
    AUTHOR = {Cullen, C. G.},
     TITLE = {An integral theorem for analytic intrinsic functions on
              quaternions},
   JOURNAL = {Duke Math. J.},
  FJOURNAL = {Duke Mathematical Journal},
    VOLUME = {32},
      YEAR = {1965},
     PAGES = {139--148},
      ISSN = {0012-7094,1547-7398},
   MRCLASS = {30.83},
  MRNUMBER = {173012},
MRREVIEWER = {R.\ F.\ Rinehart},
       URL = {http://projecteuclid.org/euclid.dmj/1077375642},
}

@article {MR3801294,
    AUTHOR = {de Fabritiis, Chiara and Gentili, Graziano and Sarfatti,
              Giulia},
     TITLE = {Quaternionic {H}ardy spaces},
   JOURNAL = {Ann. Sc. Norm. Super. Pisa Cl. Sci. (5)},
  FJOURNAL = {Annali della Scuola Normale Superiore di Pisa. Classe di
              Scienze. Serie V},
    VOLUME = {18},
      YEAR = {2018},
    NUMBER = {2},
     PAGES = {697--733},
      ISSN = {0391-173X,2036-2145},
   MRCLASS = {30G35 (30H10)},
  MRNUMBER = {3801294},
}

@article {MR4011262,
    AUTHOR = {Diki, Kamal and Gal, Sorin G. and Sabadini, Irene},
     TITLE = {Polynomial approximation in slice regular {F}ock spaces},
   JOURNAL = {Complex Anal. Oper. Theory},
  FJOURNAL = {Complex Analysis and Operator Theory},
    VOLUME = {13},
      YEAR = {2019},
    NUMBER = {6},
     PAGES = {2729--2746},
      ISSN = {1661-8254,1661-8262},
   MRCLASS = {30G35 (30E10 30H20)},
  MRNUMBER = {4011262},
MRREVIEWER = {Susheel\ Kumar},
       DOI = {10.1007/s11785-018-0878-2},
       URL = {https://doi.org/10.1007/s11785-018-0878-2},
}

@article {fueter1935die,
    AUTHOR = {Fueter, Rud},
     TITLE = {\"Uber einen {H}artogs'schen {S}atz},
   JOURNAL = {Comment. Math. Helv.},
  FJOURNAL = {Commentarii Mathematici Helvetici},
    VOLUME = {12},
      YEAR = {1939},
     PAGES = {75--80},
      ISSN = {0010-2571,1420-8946},
   MRCLASS = {30.0X},
  MRNUMBER = {693},
MRREVIEWER = {P.\ Thullen},
       DOI = {10.1007/BF01620640},
       URL = {https://doi.org/10.1007/BF01620640},
}

@article {MR2227751,
    AUTHOR = {Gentili, Graziano and Struppa, Daniele C.},
     TITLE = {A new approach to {C}ullen-regular functions of a quaternionic
              variable},
   JOURNAL = {C. R. Math. Acad. Sci. Paris},
  FJOURNAL = {Comptes Rendus Math\'ematique. Acad\'emie des Sciences. Paris},
    VOLUME = {342},
      YEAR = {2006},
    NUMBER = {10},
     PAGES = {741--744},
      ISSN = {1631-073X,1778-3569},
   MRCLASS = {30G35},
  MRNUMBER = {2227751},
       DOI = {10.1016/j.crma.2006.03.015},
       URL = {https://doi.org/10.1016/j.crma.2006.03.015},
}

@article {MR2353257,
    AUTHOR = {Gentili, Graziano and Struppa, Daniele C.},
     TITLE = {A new theory of regular functions of a quaternionic variable},
   JOURNAL = {Adv. Math.},
  FJOURNAL = {Advances in Mathematics},
    VOLUME = {216},
      YEAR = {2007},
    NUMBER = {1},
     PAGES = {279--301},
      ISSN = {0001-8708,1090-2082},
   MRCLASS = {30G35 (30B10 30C10)},
  MRNUMBER = {2353257},
MRREVIEWER = {Alessandro\ Perotti},
       DOI = {10.1016/j.aim.2007.05.010},
       URL = {https://doi.org/10.1016/j.aim.2007.05.010},
}

@article {MR3839852,
    AUTHOR = {Monguzzi, Alessandro and Sarfatti, Giulia},
     TITLE = {Shift invariant subspaces of slice {$L^2$} functions},
   JOURNAL = {Ann. Acad. Sci. Fenn. Math.},
  FJOURNAL = {Annales Academi\ae\ Scientiarum Fennic\ae. Mathematica},
    VOLUME = {43},
      YEAR = {2018},
    NUMBER = {2},
     PAGES = {1045--1061},
      ISSN = {1239-629X,1798-2383},
   MRCLASS = {30G35 (30H10 30J05)},
  MRNUMBER = {3839852},
MRREVIEWER = {M.\ Elena\ Luna-Elizarrar\'as},
       DOI = {10.5186/aasfm.2018.4366},
       URL = {https://doi.org/10.5186/aasfm.2018.4366},
}

@article {MR4162404,
    AUTHOR = {Lian, Pan and Liang, Yuxia},
     TITLE = {Weighted composition operator on quaternionic {F}ock space},
   JOURNAL = {Banach J. Math. Anal.},
  FJOURNAL = {Banach Journal of Mathematical Analysis},
    VOLUME = {15},
      YEAR = {2021},
    NUMBER = {1},
     PAGES = {Paper No. 7, 20},
      ISSN = {2662-2033,1735-8787},
   MRCLASS = {30G35 (30H20 47B33 47B38)},
  MRNUMBER = {4162404},
MRREVIEWER = {Matthew\ M.\ Jones},
       DOI = {10.1007/s43037-020-00087-6},
       URL = {https://doi.org/10.1007/s43037-020-00087-6},
}

@article {MR4519267,
    AUTHOR = {Liang, Y. and Wang, J.},
     TITLE = {Difference of quaternionic weighted composition operators on
              slice regular {F}ock spaces},
   JOURNAL = {Complex Var. Elliptic Equ.},
  FJOURNAL = {Complex Variables and Elliptic Equations. An International
              Journal},
    VOLUME = {68},
      YEAR = {2023},
    NUMBER = {1},
     PAGES = {120--134},
      ISSN = {1747-6933,1747-6941},
   MRCLASS = {30G35 (30H20 47B38)},
  MRNUMBER = {4519267},
       DOI = {10.1080/17476933.2021.1980877},
       URL = {https://doi.org/10.1080/17476933.2021.1980877},
}

@article {MR4864886,
    AUTHOR = {Liang, Guiquan and Liang, Yuxia},
     TITLE = {A class of $\mathcal{C}$-normal weighted composition operators
              on quaternionic {F}ock space},
   JOURNAL = {J. Math. Anal. Appl.},
  FJOURNAL = {Journal of Mathematical Analysis and Applications},
    VOLUME = {546},
      YEAR = {2025},
    NUMBER = {2},
     PAGES = {Paper No. 129338, 20},
      ISSN = {0022-247X,1096-0813},
   MRCLASS = {47B33},
  MRNUMBER = {4864886},
MRREVIEWER = {Mohamed\ Barraa},
       DOI = {10.1016/j.jmaa.2025.129338},
       URL = {https://doi.org/10.1016/j.jmaa.2025.129338},
}

@article {MR4384577,
    AUTHOR = {Han, Kaikai and Wang, Maofa},
     TITLE = {Slice regular weighted composition operators},
   JOURNAL = {Complex Var. Elliptic Equ.},
  FJOURNAL = {Complex Variables and Elliptic Equations. An International
              Journal},
    VOLUME = {67},
      YEAR = {2022},
    NUMBER = {1},
     PAGES = {162--223},
      ISSN = {1747-6933,1747-6941},
   MRCLASS = {30G35 (32A36 47B33)},
  MRNUMBER = {4384577},
MRREVIEWER = {M.\ Elena\ Luna-Elizarrar\'as},
       DOI = {10.1080/17476933.2020.1818731},
       URL = {https://doi.org/10.1080/17476933.2020.1818731},
}

@article {MR34064751,
    AUTHOR = {Colombo, Fabrizio and Gonz\'alez-Cervantes, J. Oscar and
              Sabadini, Irene},
     TITLE = {Further properties of the {B}ergman spaces of slice regular
              functions},
   JOURNAL = {Adv. Geom.},
  FJOURNAL = {Advances in Geometry},
    VOLUME = {15},
      YEAR = {2015},
    NUMBER = {4},
     PAGES = {469--484},
      ISSN = {1615-715X,1615-7168},
   MRCLASS = {30G35 (30H20)},
  MRNUMBER = {3406475},
MRREVIEWER = {John\ Ryan},
       DOI = {10.1515/advgeom-2015-0022},
       URL = {https://doi.org/10.1515/advgeom-2015-0022},
}

@article {MR3664521,
    AUTHOR = {Sabadini, Irene and Saracco, Alberto},
     TITLE = {Carleson measures for {H}ardy and {B}ergman spaces in the
              quaternionic unit ball},
   JOURNAL = {J. Lond. Math. Soc. (2)},
  FJOURNAL = {Journal of the London Mathematical Society. Second Series},
    VOLUME = {95},
      YEAR = {2017},
    NUMBER = {3},
     PAGES = {853--874},
      ISSN = {0024-6107,1469-7750},
   MRCLASS = {30G35 (30H10 30H20)},
  MRNUMBER = {3664521},
MRREVIEWER = {John\ Ryan},
       DOI = {10.1112/jlms.12035},
       URL = {https://doi.org/10.1112/jlms.12035},
}

@article {MR3571445,
    AUTHOR = {Sarfatti, Giulia},
     TITLE = {Quaternionic {H}ankel operators and approximation by slice
              regular functions},
   JOURNAL = {Indiana Univ. Math. J.},
  FJOURNAL = {Indiana University Mathematics Journal},
    VOLUME = {65},
      YEAR = {2016},
    NUMBER = {5},
     PAGES = {1735--1757},
      ISSN = {0022-2518,1943-5258},
   MRCLASS = {47B35 (30H10)},
  MRNUMBER = {3571445},
MRREVIEWER = {William\ Thomas\ Ross},
       DOI = {10.1512/iumj.2016.65.5896},
       URL = {https://doi.org/10.1512/iumj.2016.65.5896},
}

@article {MR3311947,
    AUTHOR = {Castillo Villalba, C. Marco Polo and Colombo, Fabrizio and
              Gantner, Jonathan and Gonz\'alez-Cervantes, J. Oscar},
     TITLE = {Bloch, {B}esov and {D}irichlet spaces of slice
              hyperholomorphic functions},
   JOURNAL = {Complex Anal. Oper. Theory},
  FJOURNAL = {Complex Analysis and Operator Theory},
    VOLUME = {9},
      YEAR = {2015},
    NUMBER = {2},
     PAGES = {479--517},
      ISSN = {1661-8254,1661-8262},
   MRCLASS = {30G35 (30H20 30H25 30H30 32A99)},
  MRNUMBER = {3311947},
MRREVIEWER = {Alessandro\ Perotti},
       DOI = {10.1007/s11785-014-0380-4},
       URL = {https://doi.org/10.1007/s11785-014-0380-4},
}

@article {MR4192830,
    AUTHOR = {Luef, Franz and Skrettingland, Eirik},
     TITLE = {A {W}iener {T}auberian theorem for operators and functions},
   JOURNAL = {J. Funct. Anal.},
  FJOURNAL = {Journal of Functional Analysis},
    VOLUME = {280},
      YEAR = {2021},
    NUMBER = {6},
     PAGES = {Paper No. 108883, 44},
      ISSN = {0022-1236,1096-0783},
   MRCLASS = {47B90 (40E05 42B10 47B10 47B35 47G30 81S10)},
  MRNUMBER = {4192830},
MRREVIEWER = {Lu\'is\ P.\ Castro},
       DOI = {10.1016/j.jfa.2020.108883},
       URL = {https://doi.org/10.1016/j.jfa.2020.108883},
}

@book {MR2934601,
    AUTHOR = {Zhu, Kehe},
     TITLE = {Analysis on {F}ock spaces},
    SERIES = {Graduate Texts in Mathematics},
    VOLUME = {263},
 PUBLISHER = {Springer, New York},
      YEAR = {2012},
     PAGES = {x+344},
      ISBN = {978-1-4419-8800-3},
   MRCLASS = {30H20 (30D15 46E20 47B10 47B35)},
  MRNUMBER = {2934601},
MRREVIEWER = {Jordi\ Pau},
       DOI = {10.1007/978-1-4419-8801-0},
       URL = {https://doi.org/10.1007/978-1-4419-8801-0},
}

@book {MR2267655,
    AUTHOR = {Bogachev, V. I.},
     TITLE = {Measure theory. {V}ol. {I}, {II}},
 PUBLISHER = {Springer-Verlag, Berlin},
      YEAR = {2007},
     PAGES = {Vol. I: xviii+500 pp., Vol. II: xiv+575},
      ISBN = {978-3-540-34513-8; 3-540-34513-2},
   MRCLASS = {28-02 (28Axx 28Cxx 46G12 60G42 60G44)},
  MRNUMBER = {2267655},
MRREVIEWER = {Ren\'e\ L.\ Schilling},
       DOI = {10.1007/978-3-540-34514-5},
       URL = {https://doi.org/10.1007/978-3-540-34514-5},
}

@article {MR3632548,
    AUTHOR = {Ghiloni, Riccardo and Perotti, Alessandro and Stoppato,
              Caterina},
     TITLE = {The algebra of slice functions},
   JOURNAL = {Trans. Amer. Math. Soc.},
  FJOURNAL = {Transactions of the American Mathematical Society},
    VOLUME = {369},
      YEAR = {2017},
    NUMBER = {7},
     PAGES = {4725--4762},
      ISSN = {0002-9947,1088-6850},
   MRCLASS = {30G35 (17D05 30C15 32A30)},
  MRNUMBER = {3632548},
MRREVIEWER = {Hennie\ De Schepper},
       DOI = {10.1090/tran/6816},
       URL = {https://doi.org/10.1090/tran/6816},
}

@book {MR2752913,
    AUTHOR = {Colombo, Fabrizio and Sabadini, Irene and Struppa, Daniele C.},
     TITLE = {Noncommutative functional calculus},
    SERIES = {Progress in Mathematics},
    VOLUME = {289},
      NOTE = {Theory and applications of slice hyperholomorphic functions},
 PUBLISHER = {Birkh\"auser/Springer Basel AG, Basel},
      YEAR = {2011},
     PAGES = {vi+221},
      ISBN = {978-3-0348-0109-6},
   MRCLASS = {30G35 (47A60)},
  MRNUMBER = {2752913},
       DOI = {10.1007/978-3-0348-0110-2},
       URL = {https://doi.org/10.1007/978-3-0348-0110-2},
}

@book {MR3059569,
     TITLE = {Advances in hypercomplex analysis},
    SERIES = {Springer INdAM Series},
    VOLUME = {1},
    EDITOR = {Gentili, Graziano and Sabadini, Irene and Shapiro, Michael and
              Sommen, Franciscus and Struppa, Daniele C.},
 PUBLISHER = {Springer, Milan},
      YEAR = {2013},
     PAGES = {viii+147},
      ISBN = {978-88-470-2444-1; 978-88-470-2445-8},
   MRCLASS = {30-06 (30G35)},
  MRNUMBER = {3059569},
       DOI = {10.1007/978-88-470-2445-8},
       URL = {https://doi.org/10.1007/978-88-470-2445-8},
}

@article {MR3358083,
    AUTHOR = {Arcozzi, Nicola and Sarfatti, Giulia},
     TITLE = {Invariant metrics for the quaternionic {H}ardy space},
   JOURNAL = {J. Geom. Anal.},
  FJOURNAL = {Journal of Geometric Analysis},
    VOLUME = {25},
      YEAR = {2015},
    NUMBER = {3},
     PAGES = {2028--2059},
      ISSN = {1050-6926,1559-002X},
   MRCLASS = {30G35 (46E22 58B20)},
  MRNUMBER = {3358083},
MRREVIEWER = {Jin-xun\ Wang},
       DOI = {10.1007/s12220-014-9503-4},
       URL = {https://doi.org/10.1007/s12220-014-9503-4},
}

@article {MR4205480,
    AUTHOR = {Kumar, Sanjay and Sharma, S. D. and Manzoor, Khalid},
     TITLE = {Quaternionic {F}ock space on slice hyperholomorphic functions},
   JOURNAL = {Filomat},
  FJOURNAL = {Univerzitet u Ni\v su. Prirodno-Matemati\v cki Fakultet.
              Filomat},
    VOLUME = {34},
      YEAR = {2020},
    NUMBER = {4},
     PAGES = {1197--1207},
      ISSN = {0354-5180,2406-0933},
   MRCLASS = {30G35 (47B38)},
  MRNUMBER = {4205480},
       DOI = {10.2298/fil2004197k},
       URL = {https://doi.org/10.2298/fil2004197k},
}

@book {MR3013643,
    AUTHOR = {Gentili, Graziano and Stoppato, Caterina and Struppa, Daniele
              C.},
     TITLE = {Regular functions of a quaternionic variable},
    SERIES = {Springer Monographs in Mathematics},
 PUBLISHER = {Springer, Heidelberg},
      YEAR = {2013},
     PAGES = {x+185},
      ISBN = {978-3-642-33870-0; 978-3-642-33871-7},
   MRCLASS = {30-02 (30B10 30C15 30C80 30E20 30G35)},
  MRNUMBER = {3013643},
MRREVIEWER = {Alessandro\ Perotti},
       DOI = {10.1007/978-3-642-33871-7},
       URL = {https://doi.org/10.1007/978-3-642-33871-7},
}

@article {MR4732451,
    AUTHOR = {Liang, Yuxia and Liu, Meicheng},
     TITLE = {Closed range and preserving frames of weighted composition
              operator on quaternionic Fock space},
   JOURNAL = {Results Math.},
  FJOURNAL = {Results in Mathematics},
    VOLUME = {79},
      YEAR = {2024},
    NUMBER = {4},
     PAGES = {Paper No. 127, 24},
      ISSN = {1422-6383,1420-9012},
   MRCLASS = {30G35 (47B38)},
  MRNUMBER = {4732451},
MRREVIEWER = {Sumit\ Kumar\ Sharma},
       DOI = {10.1007/s00025-024-02150-2},
       URL = {https://doi.org/10.1007/s00025-024-02150-2},
}

@article {MR3553407,
    AUTHOR = {Park, Jong-Do},
     TITLE = {On the transformation formula of the slice {B}ergman kernels
              in the quaternion variables},
   JOURNAL = {Bull. Korean Math. Soc.},
  FJOURNAL = {Bulletin of the Korean Mathematical Society},
    VOLUME = {53},
      YEAR = {2016},
    NUMBER = {5},
     PAGES = {1401--1409},
      ISSN = {1015-8634,2234-3016},
   MRCLASS = {30G35 (32A25 32A36)},
  MRNUMBER = {3553407},
       DOI = {10.4134/BKMS.b150720},
       URL = {https://doi.org/10.4134/BKMS.b150720},
}

@article {MR3605231,
    AUTHOR = {Park, Jong-Do},
     TITLE = {On the zeros of the slice {B}ergman kernels for the upper half
              space and the unit ball on quaternions and {C}lifford
              algebras},
   JOURNAL = {Complex Anal. Oper. Theory},
  FJOURNAL = {Complex Analysis and Operator Theory},
    VOLUME = {11},
      YEAR = {2017},
    NUMBER = {2},
     PAGES = {329--344},
      ISSN = {1661-8254,1661-8262},
   MRCLASS = {30G35 (32A25)},
  MRNUMBER = {3605231},
       DOI = {10.1007/s11785-016-0550-7},
       URL = {https://doi.org/10.1007/s11785-016-0550-7},
}

@article {MR4279369,
    AUTHOR = {Liang, Yuxia},
     TITLE = {The product operator between {B}loch-type spaces of slice
              regular functions},
   JOURNAL = {Acta Math. Sci. Ser. B (Engl. Ed.)},
  FJOURNAL = {Acta Mathematica Scientia. Series B. English Edition},
    VOLUME = {41},
      YEAR = {2021},
    NUMBER = {5},
     PAGES = {1606--1618},
      ISSN = {0252-9602,1572-9087},
   MRCLASS = {30D45 (30G35 47B33 47B91)},
  MRNUMBER = {4279369},
       DOI = {10.1007/s10473-021-0512-7},
       URL = {https://doi.org/10.1007/s10473-021-0512-7},
}

@article {MR4634680,
    AUTHOR = {Dou, Xinyuan and Ren, Guangbin and Sabadini, Irene},
     TITLE = {Extension theorem and representation formula in
              non-axially-symmetric domains for slice regular functions},
   JOURNAL = {J. Eur. Math. Soc. (JEMS)},
  FJOURNAL = {Journal of the European Mathematical Society (JEMS)},
    VOLUME = {25},
      YEAR = {2023},
    NUMBER = {9},
     PAGES = {3665--3694},
      ISSN = {1435-9855,1435-9863},
   MRCLASS = {30G35 (32A30 32D05)},
  MRNUMBER = {4634680},
MRREVIEWER = {Amedeo\ Altavilla},
       DOI = {10.4171/jems/1260},
       URL = {https://doi.org/10.4171/jems/1260},
}

@article {MR2555912,
    AUTHOR = {Colombo, Fabrizio and Gentili, Graziano and Sabadini, Irene
              and Struppa, Daniele},
     TITLE = {Extension results for slice regular functions of a
              quaternionic variable},
   JOURNAL = {Adv. Math.},
  FJOURNAL = {Advances in Mathematics},
    VOLUME = {222},
      YEAR = {2009},
    NUMBER = {5},
     PAGES = {1793--1808},
      ISSN = {0001-8708,1090-2082},
   MRCLASS = {30G35},
  MRNUMBER = {2555912},
MRREVIEWER = {Alessandro\ Perotti},
       DOI = {10.1016/j.aim.2009.06.015},
       URL = {https://doi.org/10.1016/j.aim.2009.06.015},
}

@article {MR4248852,
    AUTHOR = {Han, Kaikai and Wang, Maofa},
     TITLE = {Composition operators on generalized {F}ock spaces of slice
              hyperholomorphic functions},
   JOURNAL = {Adv. Appl. Clifford Algebr.},
  FJOURNAL = {Advances in Applied Clifford Algebras},
    VOLUME = {30},
      YEAR = {2020},
    NUMBER = {1},
     PAGES = {Paper No. 15, 16},
      ISSN = {0188-7009,1661-4909},
   MRCLASS = {47B33 (47S05)},
  MRNUMBER = {4248852},
MRREVIEWER = {Xiaomin\ Tang},
       DOI = {10.1007/s00006-020-1041-5},
       URL = {https://doi.org/10.1007/s00006-020-1041-5},
}

@article {MR3450567,
    AUTHOR = {Alpay, Daniel and Colombo, Fabrizio and Kimsey, David P.},
     TITLE = {The spectral theorem for quaternionic unbounded normal
              operators based on the {$S$}-spectrum},
   JOURNAL = {J. Math. Phys.},
  FJOURNAL = {Journal of Mathematical Physics},
    VOLUME = {57},
      YEAR = {2016},
    NUMBER = {2},
     PAGES = {023503, 27},
      ISSN = {0022-2488,1089-7658},
   MRCLASS = {47H99 (47A10)},
  MRNUMBER = {3450567},
       DOI = {10.1063/1.4940051},
       URL = {https://doi.org/10.1063/1.4940051},
}

\medskip

\noindent
School of Mathematical Sciences, Dalian University of Technology,
Dalian, Liaoning 116024, P. R. China

\noindent
Email address: \texttt{zhaopeng.lin@mail.dlut.edu.cn}

\medskip

\noindent
School of Mathematical Sciences, Dalian University of Technology,
Dalian, Liaoning 116024, P. R. China

\noindent
Email address: \texttt{lyfdlut@dlut.edu.cn}

\medskip

\noindent
School of Mathematical Sciences, Dalian University of Technology,
Dalian, Liaoning 116024, P. R. China

\noindent
Email address: \texttt{zuchao@dlut.edu.cn}
\end{document}